\documentclass{article}

\usepackage{authblk}
\usepackage{amsmath}
\usepackage{amsthm}
\usepackage{amssymb,amsfonts}
\usepackage[all,arc]{xy}
\usepackage{enumerate}
\usepackage{mathrsfs}
\usepackage{hyperref}
\usepackage[margin=1.25in]{geometry}
\usepackage{pgf,tikz,pgfplots}
\usepackage{graphicx}
\usepackage{accents}
\usepackage{mathtools}
\usepackage{url}
\usepackage{tikz-cd}
\usepackage{physics}

\newcommand{\set}[2]{\left\lbrace #1 \middle| #2 \right\rbrace}

\usepackage[textsize=scriptsize,backgroundcolor=white]{todonotes}

\tikzset{
  shaded/.style = {fill=red!10!blue!20!gray!30!white},
  unshaded/.style = {fill=white},
  shadedw/.style = {fill=white},
  shadedb/.style = {fill=blue!120!gray!30!white},
  shadedr/.style = {fill=red!120!gray!30!white},
  shadedg/.style = {fill=green!120!gray!30!white},
  shadedy/.style = {fill=yellow!120!gray!30!white},
  Tcirc/.style = {circle, draw, thick, fill=white, opaque},
  Tellip/.style = {ellipse, draw, thick, fill=white, opaque},
  Tbox/.style = {rounded corners,rectangle, draw, thick, fill=white, opaque},
	align/.style = {scale=.7,baseline},
    align.7/.style = {scale=.7, baseline},
    align1/.style = {scale=1.05, baseline},
    align1.5/.style = {scale=1.5,baseline},
  every picture/.style=semithick
}

\usetikzlibrary{
  knots,
  hobby,
  decorations.pathreplacing,
  decorations.pathmorphing,
  shapes.geometric,
  calc
}

\tikzset{
  knot diagram/every strand/.append style={black, thick},
  stock/.style={consider self intersections=true, end tolerance=1pt, clip width=5pt, clip radius=5pt},
  stockthick/.style={consider self intersections=true, end tolerance=1pt, clip width=3pt, clip radius=20pt},
  shaded/.style = {fill=red!10!blue!20!gray!30!white},
  unshaded/.style = {fill=white},
}
\pgfplotsset{compat=1.17}

\newtheorem{thm}{Theorem}[section]

\newtheorem{prop}[thm]{Proposition}
\newtheorem{lem}[thm]{Lemma}

\newtheorem{defn}[thm]{Definition}

\newtheorem{exmp}[thm]{Example}

\title{An algebra structure for reproducing kernel Hilbert spaces}
\author{Dimitrios Giannakis}
\author{Michael Montgomery\thanks{Corresponding author. Email: \href{mailto:michael.r.montgomery@dartmouth.edu}{michael.r.montgomery@dartmouth.edu}.}}
\affil{Department of Mathematics, Dartmouth College, Hanover, NH 03755, USA}
\date{}

\begin{document}
\maketitle

\begin{abstract}
    Reproducing kernel Hilbert spaces (RKHSs) are Hilbert spaces of functions where pointwise evaluation is continuous. There are known examples of RKHSs that are Banach algebras under pointwise multiplication. These examples are built from weights on the dual of a locally compact abelian group. In this paper we define an algebra structure on an RKHS that is equivalent to subconvolutivity of the weight for known examples (referred to as reproducing kernel Hilbert algebras, or RKHAs). We show that the class of RKHAs is closed under the Hilbert space tensor product and the pullback construction on the category of RKHSs. The subcategory of RKHAs becomes a monoidal category with the spectrum as a monoidal functor to the category of topological spaces. The image of this functor is shown to contain all compact subspaces of $\mathbb R^n$ for $n>0$.
\end{abstract}

\section{Introduction}

In \cite{DG23,DGM23}, the authors investigate a class of reproducing kernel Hilbert spaces (RKHS) on compact abelian groups which are Banach algebras under pointwise multiplication and an equivalent norm.
These RKHSs are built with an inverse weight function $\lambda \colon \hat{G} \to \mathbb R_{>0}$ on the (discrete) dual group $\hat G$ of a compact abelian group $G$ which is subconvolutive (see \cite{F79}).
\begin{equation} \lambda * \lambda (\gamma) \leq C \lambda(\gamma) \quad\quad \forall \gamma \in \hat{G} \tag{Subconvolutivity}
\end{equation}
It was shown that $\norm{fg}_{\mathcal H} \leq \tilde C \norm{f}_{\mathcal H}\norm{g}_{\mathcal H}$ for these examples on $G$ and any RKHS $\mathcal H$ satisfying this condition was called a reproducing kernel Hilbert algebra (RKHA). A related construction is a class of RKHSs on locally compact abelian groups studied in \cite{FeichtingerEtAl07} that generalizes the notion of harmonic Hilbert spaces on $\mathbb R$ \cite{Delvos97}. For appropriately chosen weights ($\lambda^{-1}$ is subadditive and submultiplicative) these RKHSs were shown to be Banach algebras under pointwise multiplication. RKHAs are also related to Sobolev algebras and, by Fourier duality, to weighted convolution algebras on the dual group, both of which are fields with a long history of study; e.g., \cite{Domar56,Nik70,Brandenburg75,F79,Kuznetsova06,BrunoEtAl19}. In \cite{Nik70}, an $L^p$ version of subconvolutivity on $\mathbb R^{2n}$ and the corresponding weighted $L^p$ spaces are studied. Positive weights $\nu$ on $\mathbb R^{2n}$ satisfying the subconvolutivity condition
$$C_{p,\nu}= \sup_x\left(\int_{\mathbb R^{2n}} \left(\frac{\nu(x)}{\nu(y)\nu(x-y)}\right)^{p'} dy\right)^{1/p'} < \infty \quad p \in [1, \infty], \, \frac{1}{p}+\frac{1}{p'} =1$$
called Nikolskii-Wermer weights are shown to yield Banach convolution algebras $L^p_\nu (\mathbb R^{2n})$.

In this article we present a stronger definition of RKHA which includes the examples investigated in \cite{DG23,DGM23,FeichtingerEtAl07}. We define an RKHA to be an RKHS such that pointwise multiplication extends to a bounded operator $m \colon \mathcal H \otimes \mathcal H \to \mathcal H$. This definition is compatible with the tensor product from \textbf{Hilb} and the spectrum as a functor from \textbf{BanAlg} to \textbf{Top}. Furthermore, this stronger definition is equivalent to subconvolutivity of the weight for examples built from locally compact abelian groups and a weight; see theorems~\ref{Example 1} and~\ref{Example 2}.

In the last section we consider the category of RKHAs and its compatibility with the Hilbert space tensor product, sum, pullback, and pushout constructions. The category of RKHAs is shown to be a monoidal category when equipped with the Hilbert space tensor product and the spectrum is a monoidal functor from \textbf{RKHA} to \textbf{Top}; see theorem~\ref{thm:spec}. In \cite{DGM23}, the Gelfand-Raikov-Shilov condition
\begin{equation}
    \label{eq:GRS}
    \lim_{n \to \infty} \lambda(n\gamma)^{1/n} = 1 \quad \quad \forall \gamma \in \hat{G} \tag{GRS}
\end{equation}
is necessary to show that the spectrum of $\mathcal H_\lambda$ from theorem \ref{Example 1} recovers the group $G$. In theorem~\ref{thm:GRS} this is extended to certain locally compact groups for the examples in theorem \ref{Example 2}.
We show that the Beurling-Domar condition \cite{Domar56}
\begin{equation}
    \label{eq:BD}
    \sum_{n=1}^\infty \dfrac{\ln(\lambda^{-1}(n\gamma))}{n^2} < \infty \quad \quad \forall \gamma \in \hat{G} \tag{BD}
\end{equation}
is equivalent to a spectral condition when $\mathcal H_\lambda$ is built from weights on locally compact abelian groups. Combining these results with theorem~\ref{thm:spec} it follows that the image of the spectrum functor on \textbf{RKHA} contains all compact subspaces of $\mathbb R^n$. Other results in that direction can be found in \cite{Tchamitchian84,Tchamitchian87}, who showed existence of compactly supported functions for a class of weighted $L^2$ spaces on $\mathbb R^n$ with Banach algebra structure under pointwise multiplication built from rapidly decreasing radial weights. Finally, in this paper we study aspects of Banach algebra quotients of RKHAs.

\subsection{Results from RKHS theory}

We start with basic notions in the theory of reproducing kernel Hilbert spaces. For a more detailed overview see \cite{PR16}. Letting $X$ be a topological space, $\mathcal L(X)$ will denote the collection of complex valued functions on $X$.
\begin{defn}
A function $k\colon X \times X \to \mathbb C$ is called a \textbf{positive definite kernel} if
\begin{enumerate}[(i)]
\item $k(x,y)=\overline{k(y,x)}$
\item For all $c_1, ..., c_n \in \mathbb C$ and $x_1,...,x_n \in X$ distinct, $\sum_{i,j} c_i k(x_i,x_j)\overline{c_j} \geq 0$.
\end{enumerate}
\end{defn}

A kernel is \textbf{strictly positive definite} if $\sum_{i,j} c_i k(x_i,x_j)\overline{c_j} > 0$ whenever at least one of the $c_i$ is nonzero.

A reproducing kernel Hilbert space or RKHS on a space $X$ is a vector subspace $\mathcal H \subset \mathcal L(X)$ equipped with and inner product making it a Hilbert space such that pointwise evaluation is continuous. By the Riesz representation theorem a positive definite kernel can be assigned to every RKHS.
$$ev_x(\xi)=\braket{\xi}{ev_x} \quad \quad k(x,y) = \braket{ev_y}{ev_x}$$
Furthermore, due to Moore and Aronszajn there is an equivalence between positive definite kernels and RKHSs on a space.

\begin{thm}{\cite{Aronszajn50}}
    Let $X$ be a set and $k$ a positive definite kernel on $X$. Then there exists a unique reproducing kernel Hilbert space $\mathcal H \subset \mathcal L(X)$ with $k$ as its kernel.
\end{thm}
The corresponding RKHS is commonly denoted $\mathcal H(k)$. Reproducing kernel Hilbert spaces also come with many constructions. See \cite{PR16} for more details. Let $\mathcal H_i \subset \mathcal L(X_i)$ be RKHSs with kernels $k_i$. Then
\begin{enumerate}[(i)]
\item $\mathcal H_1 \otimes \mathcal H_2 \subset \mathcal L(X_1 \times X_2)$ is an RKHS with kernel $k((x_1,x_2),(y_1,y_2))=k_1(x_1,y_1)k_2(x_2,y_2)$.
\item $\mathcal H_1 \oplus \mathcal H_2 \subset \mathcal L(X_1 \sqcup X_2)$ is an RKHS with kernel $k(x,y)=\left\lbrace\begin{array}{ccc} k_1(x,y) & \text{if } x,y \in X_1\\ k_2(x,y) & \text{if } x,y \in X_2\\ 0 & \text{else} \end{array}\right.$.
\item Let $\phi \colon S \to X$. Then $k \circ \phi$ is a positive definite kernel and the pullback $\mathcal H(k \circ \phi) = \set{f \circ \phi}{f \in \mathcal H} \subset \mathcal L(S)$ is an RKHS with $\norm{\xi}_{\mathcal H(k \circ \phi)} = \displaystyle\inf_{f, \xi=f \circ \phi} \norm{f}_\mathcal H$. When $\phi$ is an inclusion map for $S \subset X$, we denote $\mathcal H(k \circ \phi) = \mathcal H(S) \cong \overline{\text{span}\set{k_x}{x \in \phi(S)}}$. Equivalently, we have $\mathcal H(S) \cong \set{f\in \mathcal H}{f|_S = 0}^\perp$.
\item Let $\phi \colon X \to S$. Consider the closed subspace
$$\tilde{\mathcal H} = \set{f \in \mathcal H(k)}{f(x_1)=f(x_2) \text{ whenever } \phi(x_1) = \phi(x_2)}$$
and its kernel $\tilde{k} \colon X \times X \to \mathbb C$ given by the orthogonal projection onto $\tilde{\mathcal H}$, $\tilde{k}_x = P_{\tilde{\mathcal H}}(k_x)$. Then $k_\phi \colon S \times S \to \mathbb C$ by $k_\phi(s,t)=\tilde{k}(\phi^{-1}(s),\phi^{-1}(t))$ and $k_\phi(s,t) = 0$ if $\phi^{-1}(s), \phi^{-1}(t) = \emptyset$ is a well defined positive definite kernel and the pushout of $\mathcal H(k) \subset \mathcal L(X)$ along $\phi$ is given by $\mathcal H(k_\phi) \subset \mathcal L(S)$.
\item Let $\phi \colon X \to H$ where $H$ is a Hilbert space. Such a map $\phi$ is called a \textbf{feature map} and induces an RKHS $\mathcal H(\tilde k_\phi)$ on $X$ with $\tilde k_\phi \colon X \times X \to \mathbb C$ by $\tilde k_\phi(x, y) = \braket{\phi(y)}{\phi(x)}_H$.
\item If $X_1 = X_2$, then $k_1+k_2$ is a positive definite kernel and $\mathcal H(k_1+k_2) = \set{f_1+f_2}{f_i \in \mathcal H_i}$ is the pointwise sum RKHS.
\item If $X=X_1 = X_2$, let $\Delta \colon X \to X \times X$ by $\Delta(x) = (x,x)$. Then $\mathcal H_1 \odot \mathcal H_2 = \mathcal H_1 \otimes \mathcal H_2 (\Delta(X))$ is the pointwise product RKHS with kernel $k(x,y) = k_1(x,y)k_2(x,y)$.
\end{enumerate}

The pullback and pushout constructions have an unfortunate name in this context as they are not pushouts or pullbacks in the categorical sense.

\begin{defn}
Let \textbf{RKHS} denote the following category.
\begin{enumerate}[(i)]
\item Ob(\textbf{RKHS}) is the collection of reproducing kernel Hilbert spaces.
\item $T \in Mor(\mathcal H_1,\mathcal H_2)$ if $T \in B(\mathcal H_1,\mathcal H_2)$ and $T(k^1_x)=k^2_{F(x)}$ for some map $F \colon X_1 \to X_2$.
\end{enumerate}
\end{defn}

\textbf{RKHS} is equipped with two bifunctors, $\otimes$ and $\oplus$. The tensor bifunctor and identity RKHS $\mathbb C \subset \mathcal L(\{p\})$ makes $\textbf{RKHS}$ a monoidal category. The proof is identical to the proof that the category of Hilbert spaces with bounded linear maps is a monoidal category. For a definition of monoidal categories see \cite{EGNO15}. The pullback corresponds to the existence of an object, $\mathcal H(k \circ \phi)$, and an isometric morphism $T_\phi \colon \mathcal H(k \circ \phi) \to \mathcal H(k)$ by $T_\phi((k \circ \phi)_s)=k_{\phi(s)}$. Furthermore, $T_\phi$ is an unitary between $\mathcal H(k \circ \phi)$ and $\overline{span\set{k_{\phi(s)}}{s \in S}} =\set{f \in \mathcal H}{f|_{\phi(S)}=0}^\perp$. The pushout corresponds to an object, $\mathcal H(k_\phi)$, and morphism $T_\phi \colon \mathcal H(k) \to \mathcal H(k_\phi)$ by $T_\phi(k_x) = k_\phi(-,\phi(x))$. For the pushout, $T_\phi$ is a partial isometry onto $\mathcal H(k_\phi)$. Functoriality of $\otimes$ and $\oplus$ is trivial to check and the only nontrivial fact about the pullback and pushout is boundedness of $T_\phi$. For the pushout, observe that $U \colon \tilde{\mathcal H} \to \mathcal H(k_\phi)$ by $U(\tilde{k}_x) = k_\phi(-,\phi(x))$ is unitary and $T_\phi = UP_{\tilde{\mathcal H}}$.
A similar argument follows for the pullback construction.

The pointwise sum of kernels will not be considered in this paper since it is incompatible with RKHSs with an algebra structure. By contrast with the pointwise product which is built from the pullback by $\mathcal H_1 \otimes \mathcal H_2(\Delta(X))$, the pointwise sum does not appear to have an analogous construction from sums, tensor products, pullbacks and pushouts.

The main construction of RKHSs in this article will come from locally compact abelian groups. Let $G$ be a locally compact abelian group and $\mu$ a choice of Haar measure. $\hat{G}$ and $\hat{\mu}$ will denote the dual group and dual Haar measure. $\mathcal F \colon L^1(G) \to C_0(\hat{G})$ and $\hat{\mathcal F} \colon L^1(\hat{G}) \to C_0(G)$ will denote the Fourier transforms.
$$\mathcal F(f)(\gamma) = \int_G f(x)\gamma(-x)d\mu(x) \quad \quad \hat{\mathcal F}(\hat{f})(x) = \int_{\hat{G}} \hat{f}(\gamma)\gamma(x)d\hat{\mu}(x)$$

Given a locally compact group $G$, $p \in [1,\infty]$, and positive measurable weight function
$$\omega \colon \hat{G} \to \mathbb R_{>0} \quad \omega(\gamma) > \delta >0 \text{ for all } \gamma \in \hat{G}$$
we may build the weighted Wiener-Beurling $L^p$-spaces
$$\hat{\mathcal F}L^p_{\omega}(\hat{G}) = \set{\hat{\mathcal F}\hat{f} \in C_0(G)}{\hat{f} \in L^1(\hat{G}), \, \norm{f}_{\hat{\mathcal F}L^p_{\omega}(\hat{G})}=\norm{\hat{f}\omega}_{L^p(\hat{G})}<\infty}.$$
In particular, in the next section we consider the case when $p=2$, $\omega = 1/\sqrt{\lambda}$, $\lambda \in C_0(\hat{G}) \cap L^1(\hat{G})$. Then $\mathcal H_\lambda := \hat{\mathcal F}L^p_\omega(\hat{G})$ becomes a reproducing kernel Hilbert space with kernel
$$k(x,y) = \int_{\hat{G}} \lambda(\gamma) \gamma(x)\overline{\gamma(y)} d\hat{\mu}(\gamma) \quad p< \infty.$$

There is a body of literature addressing the relationship between weights $\lambda$ and properties of $\hat{\mathcal F}L^p_\omega(\hat{G})$ and other function spaces. See \cite{F79}, \cite{Grochenig07}, and \cite{Kan09} for an overview.

\section{Reproducing kernel Hilbert algebras}

In \cite{DG23}, a class of RKHSs on compact abelian groups are built such that $\norm{fg}_\mathcal H \leq \tilde{C}\norm{f}_\mathcal H \norm{g}_\mathcal H$. While this condition is desirable for an algebra structure, we also want to consider compatibility with RKHS constructions. Hilbert spaces have many special properties that will influence the algebra structure we define. First, the category of Hilbert spaces has a unique tensor product and the tensor product of two RKHAs should be an RKHA. Since morphisms in the category of Hilbert spaces are bounded linear operators, multiplication should be a bounded linear operator from $\mathcal H \otimes \mathcal H$ to $\mathcal H$. Second, since Hilbert spaces are self-dual any algebra structure also defines a coalgebra structure via the adjoint and visa versa. Having these equivalent pictures will be useful for several proofs. We begin with a few analogies with bialgebras which inspired the coalgebra approach to RKHAs.

A bialgebra over $\mathbb F$ is an $\mathbb F$ vector space, $H$, with an algebra and coalgebra structure
$$ \text{unit: } \eta \colon \mathbb F \to H \quad \text{multiplication: } \nabla \colon H \otimes H \to H$$
$$ \text{counit: } \varepsilon \colon H \to \mathbb F \quad \text{comultiplication: } \Delta \colon H \to H \otimes H$$
such that comultiplication and the counit are algebra homomorphisms. In the context of Hopf algebras, elements satisfying $\Delta(\xi) = \xi \otimes \xi$ are called group type elements.

\begin{defn}
    Let $X$ be a set and $\mathcal H \subset \mathcal L(X)$ a reproducing kernel Hilbert space of functions on $X$ with positive definite kernel. Let $\Delta \colon \{k_x \mid x \in X\} \to \mathcal H \otimes \mathcal H$ by $\Delta(k_x)=k_x \otimes k_x$. We call $\mathcal H$ a \textbf{reproducing kernel Hilbert algebra} or \textbf{RKHA} if $\Delta$ extends to a bounded linear operator on $\mathcal H$. Observe that $\Delta$ is unique if it exists since $\text{span}\set{k_x}{x \in X}$ is dense in $\mathcal H$. If $1 \in \mathcal H$, then $\mathcal H$ is called unital.
\end{defn}

If $\mathcal H$ is an RKHA, the triple $(\mathcal H, \Delta, \braket{ \cdot}{1})$ defines a coassociative and cocommutative coalgebra. Since Hilbert spaces are self-dual, these also define an algebra structure on $\mathcal H$,
$$\braket{\Delta^*(f \otimes g)}{k_x} = \braket{f\otimes g}{k_x \otimes k_x} = f(x)g(x) \quad \quad \text{for } f,g \in H,$$
$$\norm{fg}_\mathcal H = \norm{\Delta^*(f \otimes g)}_\mathcal H \leq \norm{\Delta}_{op}\norm{f}_\mathcal H\norm{g}_\mathcal H,$$
where $\norm{\cdot}_{op}$ denotes the operator norm. Equivalently, $\mathcal H$ is an RKHA if pointwise multiplication, $\Delta^* \colon \mathcal H \otimes \mathcal H \to \mathcal H$ is well defined and bounded.

If $\mathcal H$ is unital then the operator norm on $Mult(\mathcal H) \cong \mathcal H$ is equivalent to the Hilbert space norm making it a Banach algebra.
$$\norm{f}_\mathcal H =\norm{M_f 1}_\mathcal H \leq \norm{M_f}_{op} \norm{1}_\mathcal H$$
$$\frac{1}{\norm{1}_\mathcal H} \norm{f}_\mathcal H \leq \norm{M_f}_{op} \leq \norm{\Delta}_{op} \norm{f}_\mathcal H$$
In the nonunital case the norm on $\mathcal H$ can be scaled to make $(\mathcal H, \Delta^*)$ into a Banach algebra $\norm{\cdot}_{Ban} = \norm{\Delta}_{op} \norm{\cdot}_{\mathcal H}$. Therefore $(\mathcal H,\Delta)$ is always a Banach algebra with respect to a norm equivalent to the Hilbert space norm. Let $\norm{\cdot}_{Ban}$ denote the rescaled norm in the nonunital case and the operator norm in the unital case. This is by no means a unique choice, however, the norm and weak-$*$ topologies on $\mathcal H$ remain unchanged. In the unital case the norms $\norm{\cdot}_{Ban}$ and $\norm{\cdot}_{\mathcal H}$ must be different as $\mathbb C$ is the only Hilbert algebra with a submultiplicative norm and unit norm multiplicative identity (see \cite{Cho78}).

Strict positive definiteness of the kernel implies that $\mathcal H$ separates the points of $X$, and $\Delta$ extends to a densely defined linear operator on $span \{ k_x \mid x \in X \}$ even if $\mathcal H$ is not an RKHA. Boundedness of the kernel is another immediate consequence of boundedness of $\Delta$.
\begin{equation}
    \label{eq:sqrt_kx}
    \sqrt{k(x,x)}=\frac{\norm{\Delta(k_x)}_\mathcal H}{\norm{k_x}_\mathcal H} \leq \norm{\Delta}_{op}
\end{equation}
The existence of $\Delta$ as a bounded map also has interesting consequences for the morphisms between RKHAs. Let $\mathcal H_i$ be unital RKHAs and $T \colon \mathcal H_1 \to \mathcal H_2$ a morphism. Since (as one can verify from the definitions) $\Delta T = (T \otimes T) \Delta$, we have $T^* \Delta^* = \Delta^* (T^* \otimes T^*)$ and so $T^*$ is multiplicative. Moreover, since  $T(k^1_x) =k^2_{F(x)}$ for all $x \in X$, we have $\braket{T^*(1_{\mathcal H_2})}{k_x^1} =1$ which implies $T^*(1_{\mathcal H_2})=1_{\mathcal H_1}$. It follows that $T^*$ is a morphism between unital Banach algebras. Furthermore, $T^*$ being a unital Banach algebra morphism implies $T$ is a morphism between RKHSs.

The reproducing kernel Hilbert algebras built in \cite{DG23} have a weaker definition than the one above. We begin by showing that the weighted Wiener-Beurling type algebras with subconvolutive weights built in \cite{DG23} satisfy this stronger definition as well. We address the compact and locally compact cases separately.

\begin{thm}\label{Example 1}
Let $G$ be a compact abelian group and let $\lambda \in L^1(\hat{G})$ be strictly positive. Then
$$\mathcal H=\set{\hat{\mathcal F}\hat{f} \in C(G)}{\sum_{\gamma \in \hat{G}} \frac{\abs{\hat{f}(\gamma)}^2}{\lambda(\gamma)} < \infty} \quad \quad \braket{f}{g}_\mathcal H = \sum_{\gamma \in \hat{G}} \frac{\hat{f}(\gamma) \overline{\hat{g}(\gamma)}}{\lambda(\gamma)}$$
$$k_x(y) = \sum_{\gamma \in \hat{G}} \lambda(\gamma) \overline{\gamma(x)} \gamma(y)$$
is an RKHS with a strictly positive definite kernel function. Furthermore, $\mathcal H$ is an RKHA iff $\lambda$ is subconvolutive (i.e. $\lambda * \lambda (\gamma) \leq C \lambda(\gamma)$).
\end{thm}

\begin{proof}
Let $\lambda$ be subconvolutive. We will show that $\Delta$ is bounded by diagonalizing it over the orthonormal bases of $\mathcal H$ and $\mathcal H \otimes \mathcal H$ coming from characters of $G$. Define the orthonormal basis
$$\set{ \psi_\gamma = \sqrt{\lambda(\gamma)}\gamma}{\gamma \in \hat{G}} \subset \mathcal H$$
and $T \colon \mathcal H \to \mathcal H \otimes \mathcal H$ by
$$T(\psi_\gamma) = \sum_{\alpha+\beta=\gamma} \sqrt{\frac{\lambda(\alpha)\lambda(\beta)}{\lambda(\gamma)}} \psi_\alpha \otimes \psi_\beta.$$
$T$ is bounded iff $\lambda$ is subconvolutive since
$$\braket{T(\psi_\gamma)}{T(\psi_\sigma)} = \delta_{\gamma=\sigma} \sum_{\alpha+\beta=\gamma} \frac{\lambda(\alpha)\lambda(\beta)}{\lambda(\gamma)} =\delta_{\gamma=\sigma} \frac{\lambda * \lambda (\gamma)}{\lambda(\gamma)}.$$
Observe that $k_x=\sum_{\gamma \in \hat{G}} \overline{\psi_\gamma(x)} \psi_\gamma$, and so
\begin{equation}
    \label{eq:T_kx}
    T(k_x)=\sum_{\gamma \in \hat{G}} \sum_{\alpha+\beta=\gamma} \sqrt{\frac{\lambda(\alpha)\lambda(\beta)}{\lambda(\gamma)}} \overline{\psi_{\gamma}(x)} \psi_\alpha \otimes \psi_\beta = \sum_{\alpha,\beta \in \hat{G}} \overline{\psi_\alpha(x) \psi_\beta(x)} \psi_\alpha \otimes \psi_\beta = k_x \otimes k_x.
\end{equation}
Therefore $T$ extends $\Delta$ to a bounded linear operator. For completeness, we note that
\begin{equation}
    \label{eq:T_unital}
    T(\xi) = \sum_{(\alpha,\beta) \in \hat G \times \hat G} \dfrac{\hat{\xi}(\alpha+\beta)}{\sqrt{\lambda(\alpha+\beta)}} \sqrt{\dfrac{\lambda(\alpha)\lambda(\beta)}{\lambda(\alpha+\beta)}} \psi_\alpha \otimes \psi_\beta \quad \quad \text{for } \xi = \hat{\mathcal F}(\hat \xi) \in \mathcal H.
\end{equation}

Next, let $\Delta$ be bounded such that $\Delta(k_x)=k_x \otimes k_x$. Since
$$\psi_\gamma = \int_G k_x \frac{\gamma(x)}{\sqrt{\lambda(\gamma)}} d\mu(x)$$
where $\mu$ is the Haar measure on $G$, then
$$\Delta(\psi_\gamma) = \int_G k_x \otimes k_x \frac{\gamma(x)}{\sqrt{\lambda(\gamma)}} d\mu(x)=\sum_{\alpha+\beta=\gamma} \sqrt{\frac{\lambda(\alpha)\lambda(\beta)}{\lambda(\gamma)}} \psi_\alpha \otimes \psi_\beta.$$
We observed above that this operator is bounded iff $\lambda$ is subconvolutive.
\end{proof}

\begin{thm}\label{Example 2}
Let $G$ be a locally compact abelian group and $\lambda \in L^1(\hat{G})\cap C_0(G)$ strictly positive. Then
$$\mathcal H = \set{\hat{\mathcal F} \hat f \in C_0(G)}{\int_{\hat{G}} \frac{\abs{\hat{f}(\gamma)}^2}{\lambda(\gamma)} d\hat{\mu}(\gamma) < \infty} \quad \braket{f}{g}_\mathcal H = \int_{\hat{G}} \frac{\hat{f}(\gamma)\overline{\hat{g}(\gamma)}}{\lambda(\gamma)} d\hat{\mu}(\gamma)$$
$$k_x(y)=\int_{\hat{G}}\lambda(\gamma)\overline{\gamma(x)}\gamma(y)d\hat{\mu}(\gamma)$$
is an RKHA iff $\lambda$ is subconvolutive.
\end{thm}

Observe that $\hat{f} \in L^1(\hat{G})$ is automatic from $\int_{\hat{G}} \frac{\abs{\hat{f}(\gamma)}^2}{\lambda(\gamma)} d\hat{\mu}(\gamma) < \infty$ since $\norm{\hat{f}}_{L^1(\hat{G})} =\norm{\frac{\hat{f}}{\sqrt{\lambda}} \sqrt{\lambda}}_{L^1(\hat{G})} \leq \norm{\sqrt{\lambda}}_{L^2(\hat{G})} \int_{\hat{G}} \frac{\abs{\hat{f}(\gamma)}^2}{\lambda(\gamma)} d\hat{\mu}(\gamma) <\infty$. Strict positivity of $\lambda \in L^1(\hat{G})$ also forces $\sigma$-finiteness of $(\hat{G},\hat{\mu})$.

\begin{lem}
    With the notation of theorem~\ref{Example 2}, $\Delta \colon \text{span}\set{k_x}{x \in \hat{G}} \to \mathcal H \otimes \mathcal H$ is well-defined and closable for strictly positive $\lambda \in L^1(\hat{G}) \cap C_0(\hat G)$.
\end{lem}

\begin{proof}
    We start with a densely defined operator $T \colon \mathcal H \to \mathcal H \otimes \mathcal H$ that will later coincide with comultiplication.
Define $\psi_\gamma = \sqrt{\lambda(\gamma)} \gamma \in C_b(G,\mathbb C)$ and
$$\zeta_\gamma(x,y) = \int_{\hat{G}} \sqrt{\frac{\lambda(\alpha)\lambda(\gamma-\alpha)}{\lambda(\gamma)}} \psi_\alpha(x) \psi_{\gamma-\alpha}(y)d\hat{\mu}(\alpha) \in C_b(G\times G, \mathbb C).$$
Since $\lambda \in L^1(\hat G) \cap C_0(\hat G) $, for every $\gamma \in \hat G$ and $x,y \in G$, the functions $\alpha \mapsto \sqrt{\lambda(\alpha)\lambda(\gamma-\alpha)}$ and $\alpha \mapsto \psi_\alpha(x)\psi_{\gamma-\alpha}(y)$ lie in $L^2(\hat G)$. Thus by Cauchy-Schwartz
$$\abs{\zeta_\gamma(x,y)} = \abs{\braket{\sqrt{\frac{\lambda(\alpha)\lambda(\gamma-\alpha)}{\lambda(\gamma)}}}{\psi_\alpha(x) \psi_{\gamma-\alpha}(y)}_{L^2(\hat{G})}} \leq \sqrt{\lambda(\gamma)} \frac{(\lambda * \lambda)(\gamma)}{\lambda(\gamma)}.$$ The following dense subspace of $\mathcal H$ will be the domain of $T$
$$\mathcal D(T)=\set{\hat{\mathcal F}\hat \xi}{\hat{\xi} \in L^1(\hat G) \cap L^\infty(\hat{G}), \text{supp}(\hat{\xi}) \subseteq E \subset \hat{G} \text{ with } E \text{ compact}}\subset \mathcal H.$$
Since $\lambda|_E$ is bounded below by a strictly positive number, $T \colon \mathcal D(T) \to \mathcal H \otimes \mathcal H$ given by
$$T(\xi)(x,y) = \int_{\hat{G}} \frac{\hat{\xi}(\gamma)}{\sqrt{\lambda(\gamma)}}\zeta_\gamma(x,y)d\hat{\mu}(\gamma)$$
is well-defined. The integrand is absolutely integrable for every $(x,y) \in G \times G$ since $$\gamma \mapsto \abs{\frac{\hat{\xi}(\gamma)}{\sqrt{\lambda(\gamma)}}\zeta_\gamma(x,y)} \leq \abs{\hat{\xi}(\gamma)} \frac{(\lambda * \lambda)(\gamma)}{\lambda(\gamma)}$$ and $\hat{\xi} \in L^1(\hat{G})$ has compact support. By Fubini-Tonelli and a change variables
\begin{equation}
    \label{eq:T_nonunital}
T(\xi)(x,y) = \iint_{\hat{G}\times \hat{G}} \dfrac{\hat{\xi}(\alpha+\beta)}{\sqrt{\lambda(\alpha+\beta)}} \sqrt{\dfrac{\lambda(\alpha)\lambda(\beta)}{\lambda(\alpha+\beta)}} \psi_\alpha(x) \psi_\beta(y) d\hat{\mu}\times\hat{\mu}(\alpha,\beta).
\end{equation}
Note that $T(\xi)$ above reduces to the formula~\eqref{eq:T_unital} for the unital case when $\hat G$ is a discrete group and $\hat \mu$ the counting measure.
To show that $T(\mathcal D(T)) \subset \mathcal H \otimes \mathcal H$ it suffices that $(\alpha, \beta) \mapsto \dfrac{\hat{\xi}(\alpha+\beta)}{\sqrt{\lambda(\alpha+\beta)}} \sqrt{\dfrac{\lambda(\alpha)\lambda(\beta)}{\lambda(\alpha+\beta)}} \in L^2(\hat{G} \times \hat{G})$ for $\xi \in \mathcal D(T)$. This is easily verified as compact support guarantees
$$\iint_{\hat{G}\times \hat{G}} \frac{\lambda(\alpha)\lambda(\gamma-\alpha)}{\lambda(\gamma)} \frac{\abs{\hat{\xi}(\gamma)}^2}{\lambda(\gamma)} d\hat{\mu}\times \hat{\mu}(\alpha,\gamma) < \infty.$$
Finally, we show that $T \colon \mathcal D(T) \to \mathcal H \otimes \mathcal H$ is closable with $\hat{\mathcal F}[C_c(\hat{G} \times \hat{G})] \subset \mathcal D(T^*)$. Let $E \subset \hat{G}$ be the support of $\hat{\eta} \in C_c(\hat{G} \times \hat{G})$, $\eta = \hat{\mathcal F}(\hat{\eta})$. Then by definition
$$\abs{\braket{T(\xi)}{\eta}_{\mathcal H \otimes \mathcal H}} = \abs{\iint_{\hat{G}\times \hat{G}} \dfrac{\hat{\xi}(\alpha+\beta)}{\sqrt{\lambda(\alpha+\beta)}} \sqrt{\dfrac{\lambda(\alpha)\lambda(\beta)}{\lambda(\alpha+\beta)}} \dfrac{\hat{\eta}(\alpha,\beta)}{\sqrt{\lambda(\alpha)\lambda(\beta)}}d\hat{\mu}\times \hat{\mu}(\alpha,\beta)}$$
$$\leq \iint_{\hat{G}\times \hat{G}} \abs{\dfrac{\hat{\xi}(\gamma)}{\sqrt{\lambda(\gamma)}}} \abs{\dfrac{\hat{\eta}(\alpha,\gamma-\alpha)}{\sqrt{\lambda(\gamma)}}}d\hat{\mu}\times \hat{\mu}(\gamma,\alpha) \leq \norm{\xi}_\mathcal H \int_{\hat{G}} \left(\int_{\hat{G}} \abs{\dfrac{\hat{\eta}(\alpha,\gamma-\alpha)}{\sqrt{\lambda(\gamma)}}}^2d\hat{\mu}(\gamma)\right)^{1/2} d\hat{\mu}(\alpha)$$
$$\leq \norm{\xi}_\mathcal H \norm{\hat{\eta}}_\infty \sup_{(\alpha,\beta) \in E} \frac{1}{\sqrt{\lambda(\alpha+\beta)}} \int_{\hat{G}} \left(\int_{\hat{G}} 1_{E}(\alpha,\beta)d\hat{\mu}(\beta)\right)^{1/2}d\hat{\mu}(\alpha) \leq const(\hat{\eta}) \norm{\xi}_\mathcal H$$
since $E$ is compact and $\lambda$ is strictly positive and continuous. Density of $\hat{\mathcal F}[C_c(\hat{G} \times \hat{G})] \subset \mathcal H \otimes \mathcal H$ implies that $T$ is closable.

We now use $T$ to show that $\Delta$ is closable. For $E \subset \hat{G}$ compact and $x \in G$ define $k_{x,E} = \int_E \overline{\psi_\gamma(x)} \psi_\gamma d\hat{\mu}(\gamma) \in \mathcal D(T)$.
Then similar to the computation of $T(k_x)$ in \eqref{eq:T_kx},
$$T(k_{x,E}) = \int_{\hat{G}} \int_{E} \overline{\psi_\alpha(x)} \psi_\alpha \otimes \overline{\psi_{\gamma-\alpha}(x)} \psi_{\gamma-\alpha} d\hat{\mu}(\gamma)d\hat{\mu}(\alpha)$$ and
$$\norm{k_{x,E}-k_x}^2_\mathcal H = \int_{\hat{G}\backslash E} \lambda(\gamma) d\hat{\mu}(\gamma)$$
$$\norm{T(k_{x,E}) - k_x \otimes k_x}^2_{\mathcal H \otimes \mathcal H} = \int_{\hat{G}\backslash E} \int_{\hat{G}} \lambda(\alpha)\lambda(\gamma-\alpha) d\hat{\mu}(\alpha) d\hat{\mu}(\gamma)=\int_{\hat{G}\backslash E} (\lambda *\lambda)(\gamma) d\hat{\mu}(\gamma).$$
Let $\Lambda$ be the directed set of compact subsets of $\hat{G}$ under inclusion and consider the net,\\ $\left(k_{x,E},T(k_{x,E})\right)_{E \in \Lambda}$ in the graph of $T$. We must show that this net converges to $(k_x,k_x \otimes k_x)$.
Consider the open subsets $F_n = \lambda^{-1}\left(\frac{1}{n},\infty\right) \subset \hat{G}$. Since $\lambda \in L^1(\hat{G})$, $\hat{\mu}(F_n) < \infty$ and for all $\varepsilon >0$ there exists $n_0$ such that $\int_{\hat{G}\backslash F_{n_0}} \lambda d\hat{\mu} < \varepsilon/2$. By inner regularity of the Haar measure $\hat{\mu}(F_{n_0}) = \sup\set{\hat{\mu}(K)}{K \subset F_{n_0}}$, and so there exist $E \subset F_{n_0}$ compact such that $\int_{F_{n_0} \backslash E}\lambda d\hat{\mu} <\varepsilon/2$ which implies $\int_{\hat{G}\backslash E} \lambda d\hat{\mu} < \varepsilon$. Since $\lambda * \lambda \in L^1(\hat{G})$, is strictly positive, and continuous, the argument above implies convergence of the net $\left(k_{x,E},T(k_{x,E})\right)_{E \in \Lambda}$ to $\left( k_x, k_x \otimes k_x \right)$. Finally, let $\tilde T: D(\tilde T) \to \mathcal H \otimes \mathcal H$ be a closed extension of $T$. Since $\tilde T$ is closed and $\lim ( k_{x,E}, \tilde T(k_{x,E})) = \lim ( k_{x,E}, T(k_{x,E})) = (k_x, \Delta k_x)$ we conclude that $\Delta \subseteq \tilde T$ so $\Delta$ is closable.
\end{proof}

We now prove theorem \ref{Example 2}.
\begin{proof}[Proof of theorem]
    Suppose that $\lambda$ is subconvolutive with $\frac{(\lambda * \lambda) (\gamma)}{\lambda(\gamma)} \leq C$. Then $\tilde T\supseteq \Delta$ above is a bounded operator on $\mathcal H$ with $\norm{\tilde T}_{op} \leq \sqrt{C}$ as
$$\norm{T(\xi)}^2_{\mathcal H \otimes \mathcal H} =\iint_{\hat{G}\times \hat{G}} \frac{\lambda(\alpha)\lambda(\gamma-\alpha)}{\lambda(\gamma)} \frac{\abs{\hat{\xi}(\gamma)}^2}{\lambda(\gamma)} d\hat{\mu}\times \hat{\mu}(\alpha,\gamma)=\int_{\hat{G}} \frac{(\lambda*\lambda)(\gamma)}{\lambda(\gamma)} \frac{\abs{\hat{\xi}(\gamma)}^2}{\lambda(\gamma)} d\hat{\mu}(\gamma) \leq C \norm{\xi}^2_\mathcal H.$$

For the converse, suppose that $\Delta \colon \mathcal H \to \mathcal H \otimes \mathcal H$ is a bounded operator such that $\Delta(k_x) = k_x \otimes k_x$. Then immediately $\overline{T} = \Delta$ is bounded. Since $T$ is bounded and
$$\norm{T(\xi)}^2_{\mathcal H \otimes \mathcal H} =\int_{\hat{G}} \frac{(\lambda*\lambda)(\gamma)}{\lambda(\gamma)} \frac{\abs{\hat{\xi}(\gamma)}^2}{\lambda(\gamma)} d\hat{\mu}(\gamma)=\braket{M_{\frac{\lambda*\lambda}{\lambda}} \frac{\hat{\xi}}{\sqrt{\lambda}}}{\frac{\hat{\xi}}{\sqrt{\lambda}}}_{L^2(\hat{G})}$$
it is necessary for $\frac{\lambda*\lambda}{\lambda}$ to be bounded. Hence $\norm{\Delta}_{op} = \sqrt{\norm{M_{\frac{\lambda*\lambda}{\lambda}}}_{op}} = \sqrt{C}$ where $C$ is the smallest constant such that $\frac{\lambda*\lambda}{\lambda}(\gamma) \leq C$ for all $\gamma \in \hat{G}$.
\end{proof}

\begin{exmp}\label{subexponential weights}
Examples of subconvolutive weights can be found in \cite{F79} and \cite{CNW73}. The weights we consider are
\begin{enumerate}[(i)]
\item $\hat{G} = \mathbb Z^n$ and $\lambda(k) = e^{-\tau \norm{k}_p^p}$ for $0<\tau$ and $0<p<1$.
\item $\hat{G} = \mathbb R^n$ and $\lambda(x) = e^{-\tau \norm{x}_p^p}$ for $0<\tau$ and $0<p<1$.
\end{enumerate}
where $\norm{x}_{p}^p=\sum_{i=1}^n x_i^p$. These specific weights also satisfy the GRS and BD conditions which will be important later. The weight $(i)$ is also submultiplicative.
\end{exmp}

We now give a weaker version of bounded approximate unit that is compatible with the categorical structure we will investigate later.

\begin{defn}
A nonunital RKHA $\mathcal H$ has a weak approximate unit if there is a net $(\eta_i)_{i \in I} \subset \mathcal H$ such that $M_{\eta_i} \to id_{\mathcal H}$ in the weak operator topology and $(M_{\eta_i})_{i \in I}$ is operator norm bounded.
\end{defn}

Consider the example $\mathcal H_\lambda \subset C_0(\mathbb R)$ with $\lambda(x) = e^{-\tau \abs{x}^p}$, $\tau >0$, $0<p<1$. We now show that $\eta_n = sinc(\frac{x}{2n})$ is a countable weak approximate unit by reducing the problem to the existence of an approximate unit for $L^1(\mathbb R)$ under convolution. We will work entirely in the inverse Fourier domain with $\hat{\eta_n} = n 1_{[-1/2n,1/2n]}$.

\begin{proof} Observe that $\frac{\lambda(\gamma+\delta)}{\lambda(\gamma)} \leq e^{\tau \abs{\delta}^p}$. Then by Cauchy-Schwarz and Holders inequalities
$$n^2\int_{-1/2n}^{1/2n}\int_{-1/2n}^{1/2n}\int_{\mathbb R} \abs{\frac{\hat{\xi}(\gamma+\delta)}{\sqrt{\lambda(\gamma+\delta)}}\frac{\overline{\hat{\xi}(\gamma+\varepsilon)}}{\sqrt{\lambda(\gamma+\varepsilon)}}} \frac{\sqrt{\lambda(\gamma+\delta)}\sqrt{\lambda(\gamma+\varepsilon)}}{\lambda(\gamma)} d\gamma d\delta d\varepsilon \leq e^{\tau (1/2n)^p} \norm{\xi}_{\mathcal H}^2.$$
By Fubini-Tonelli and a change of variables the integral above majorizes
$$\norm{\eta_n \xi}_{\mathcal H}^2 = \int_{\mathbb R} \abs{n\int_{\gamma-1/2n}^{\gamma+1/2n} \hat{\xi}(\delta)d\delta}^2 \frac{1}{\lambda(\gamma)} d\gamma \leq e^{\tau (1/2n)^p} \norm{\xi}_{\mathcal H}^2.$$

In what follows, we must take local averages of functions $\frac{\hat{\xi}}{\sqrt{\lambda}} \in L^1_{loc}(\mathbb R) \cap L^2(\mathbb R)$ and verify that $\left(\frac{\hat{\xi}}{\sqrt{\lambda}}\right)_n = \left[ \gamma \mapsto n \int_{\gamma-1/2n}^{\gamma+1/2n}\frac{\hat{\xi}(\delta)}{\sqrt{\lambda(\delta)}} d\delta\right] \in L^1_{loc}(\mathbb R) \cap L^2(\mathbb R)$. The argument above implies these local averages are in $L^2(\mathbb R)$ and $\norm{\left(\frac{\hat{\xi}}{\sqrt{\lambda}}\right)_n}_{L^2(\mathbb R)} \leq \norm{\frac{\hat{\xi}}{\sqrt{\lambda}}}_{L^2(\mathbb R)}$. It follows easily that $\left(\frac{\hat{\xi}}{\sqrt{\lambda}}\right)_n \in L^1_{loc}(\mathbb R)$. Then
$$\abs{\braket{\eta_n \xi}{\zeta}_{\mathcal H}-\braket{\left(\frac{\hat{\xi}}{\sqrt{\lambda}}\right)_n}{\frac{\hat{\zeta}}{\sqrt{\lambda}}}_{L^2(\mathbb R)}} =\abs{\int_{\mathbb R} \frac{\overline{\hat{\zeta}(\gamma)}}{\sqrt{\lambda(\gamma)}} n \int_{\gamma-1/2n}^{\gamma +1/2n} \frac{\hat{\xi}(\delta)}{\sqrt{\lambda(\delta)}} \left( \sqrt{\frac{\lambda(\delta)}{\lambda(\gamma)}} -1 \right) d\delta d \gamma}$$
$$\leq max\left\lbrace\abs{e^{(\tau/2)(1/2n)^p} -1},\abs{e^{-(\tau/2)(1/2n)^p} -1}\right\rbrace \norm{\xi}_{\mathcal H} \norm{\zeta}_{\mathcal H}$$
converges to zero as $n$ goes to infinity. This reduces the problem of showing $\abs{\braket{\eta_n \xi}{\zeta}_{\mathcal H} - \braket{\xi}{\zeta}_{\mathcal H}} \to 0$ to proving $\left(\frac{\hat{\xi}}{\sqrt{\lambda}}\right)_n \to \frac{\hat{\xi}}{\sqrt{\lambda}}$ weakly in $L^2(\mathbb R)$. Since this sequence is norm bounded it suffices to check the condition for weak convergence against a dense subspace of $L^2(\mathbb R)$. Take $\frac{\hat{\zeta}}{\sqrt{\lambda}}$ to be the indicator function on a bounded interval $E$ and observe that $\left(\frac{\hat{\xi}}{\sqrt{\lambda}}\right)_n, \frac{\hat{\xi}}{\sqrt{\lambda}} \in L^1_{loc}(\mathbb R)$
$$\abs{\braket{\left(\frac{\hat{\xi}}{\sqrt{\lambda}}\right)_n}{1_E}_{L^2(\mathbb R)} - \braket{\frac{\hat{\xi}}{\sqrt{\lambda}}}{1_E}_{L^2(\mathbb R)}} \leq \norm{\left(\frac{\hat{\xi}}{\sqrt{\lambda}}\right)_n - \frac{\hat{\xi}}{\sqrt{\lambda}}}_{L^1(E)}$$
which converges to zero.
\end{proof}

This approximate unit also gives an example of an RKHA where the operator norm and Hilbert space norms are not equivalent. Suppose that $\norm{\cdot}_{\mathcal H_\lambda}$ and $\norm{\cdot}_{op}$ are equivalent in the example above. Then $\eta_n$ is $\norm{\cdot}_{\mathcal H_\lambda}$ bounded and hence has a weakly convergent subsequence $\eta_{n_k} \to \eta$ with the property
$$k(x,x) =\lim_{k \to \infty} \braket{\eta_{n_k}k_x}{k_x} =\lim_{k \to \infty} \braket{\eta_{n_k}\otimes k_x }{k_x\otimes k_x}  = \braket{\eta}{k_x}k(x,x).$$
Therefore, $\eta(x) =1$ for all $x \in G$ and $\mathcal H_\lambda$ is unital which is a contradiction.

We end this section with some remarks on the spectra of RKHAs that will be useful in the categorical constructions of section~\ref{sec:rkha_cat}. We employ the standard definition from abelian Banach algebras to define the spectrum, $\sigma(\mathcal H)$, of an RKHA $\mathcal H$ on a set $X$ as the set of nonzero, multiplicative linear functionals on $\mathcal H$ called characters, equipped with the weak-$*$ topology from $\mathcal H$. Continuity of characters is automatic since they can be extended to characters on the unitalization. It is immediate that $\sigma(\mathcal H)$ contains all nonzero evaluation functionals $ev_x$, $x \in X$. This implies that the spectrum separates points in $\mathcal H$ and we can canonically identify $\mathcal H$ with a subspace of $C(\sigma(\mathcal H))$ under the Gelfand transform, $f \in \mathcal H \mapsto \tilde f \in C(\sigma(\mathcal H))$ with $\tilde f(\phi) = \phi(f)$. In what follows, we will make that identification when convenient.

\section{The RKHA category}
\label{sec:rkha_cat}

Let $\textbf{RKHA}$ denote the full subcategory of $\textbf{RKHS}$ whose objects are RKHAs. There are four constructions we consider in relation to this category, $\oplus$, $\otimes$, the pullback, and the pushout. In this section we will show that the following properties are compatible with the four constructions considered.
\begin{enumerate}[(i)]
\item $\mathcal H$ is unital/nonunital.
\item For all $f \in \mathcal H$, $f \colon X \to \mathbb C$ is continuous.
\item $\mathcal H \subset  C(\sigma(\mathcal H))$ has functions with arbitrarily small compact support.
\end{enumerate}
Finally, we address the compatibility of the spectrum of $\mathcal H$ with the various constructions and properties listed above.

\begin{prop}\label{RKHA closed under operations}
\textbf{RKHA} is closed under $\otimes$, $\oplus$, and it contains pullbacks and pushouts. Additionally, unital RKHAs, denoted as $\textbf{RKHA}_{\textbf{u}}$, are closed under these operations.
\end{prop}

\begin{proof}
To show that $\mathcal H_1 \otimes \mathcal H_2$ is an RKHA, observe that $(id_{\mathcal H_1} \otimes \tau \otimes id_{\mathcal H_2}) \circ \Delta_1 \otimes \Delta_2 (k_x \otimes k_y)  = (k_x \otimes k_y) \otimes (k_x \otimes k_y)$ is a bounded map where $\tau \colon \mathcal H_1 \otimes \mathcal H_2 \to \mathcal H_2 \otimes \mathcal H_1$ is the unitary induced by $\tau(\xi \otimes \eta) = \eta \otimes \xi$.

Trivially, $H_1 \oplus H_2$ has a bounded coproduct given by $\Delta_1 \oplus \Delta_2$.

Let $\phi \colon S \to X$ with $\mathcal H \subset \mathcal L(X)$ and RKHA. Since $T_\phi \colon \mathcal H(k \circ \phi) \to \mathcal H(k)$ is an isometry, it has a bounded left inverse $T_\phi^*$, $T_\phi^*T_\phi = id_{\mathcal H(k \circ \phi)}$. Then by inspection $(T_\phi^* \otimes T_\phi^*)\Delta T_\phi \colon \mathcal H(k \circ \phi) \to \mathcal H(k \circ \phi) \otimes \mathcal H(k \circ \phi)$ is bounded and satisfies the coproduct property.

Let $\phi \colon X \to S$ with $\mathcal H \subset \mathcal L(X)$ an RKHA. Observe that $$\tilde{\mathcal H} = \set{f \in \mathcal H(k)}{f(x_1)=f(x_2) \text{ whenever } \phi(x_1) = \phi(x_2)}$$ is a closed subalgebra of $\mathcal H$. Therefore $\tilde{\Delta}^*=\Delta^*|_{\tilde{\mathcal H}\otimes \tilde{\mathcal H}} \colon \tilde{\mathcal H} \otimes \tilde{\mathcal H} \to \tilde{\mathcal H}$ is bounded and implements pointwise multiplication. As a result $\tilde{\Delta}(\tilde{k}_x) = \tilde{k}_x \otimes \tilde{k}_x$ and so $\tilde{\mathcal H}\subset \mathcal L(X)$ is an RKHA. Finally, $T_\phi \colon \tilde{\mathcal H} \to \mathcal H(k_\phi)$ given by $U_\phi(\tilde{k}_x) = k_\phi(-,\phi(x))$ defines a unitary. Hence, $(U_\phi\otimes U_\phi)\tilde{\Delta}U_{\phi}^*$ implements the coproduct on $\mathcal H(k_\phi)$.

For unital RKHAs, clearly $1_{\mathcal H_1} \oplus 1_{\mathcal H_2}$ and $1_{\mathcal H_1} \otimes 1_{\mathcal H_2}$ are the units for $\mathcal H_1 \oplus \mathcal H_2$ and $\mathcal H_1 \otimes \mathcal H_2$ respectively. The pushout is unital since the unit trivially belongs to $\tilde{\mathcal H}$. For the pullback, $T_\phi^*(1)$ is the unit since $T_\phi^*$ is multiplicative and $T_\phi$ is an isometry
$$M_{T_\phi^*(1)}(f) = T_\phi^*(1)f = T_\phi^*(1) \cdot T_\phi^*T_\phi(f) = T_\phi^*(1 \cdot T_\phi(f)) = f \quad \text{for } f \in \mathcal H(k \circ \phi).$$
\end{proof}

We can also characterize the subcategory \textbf{RKHA} by the existence of a natural transformation, $\Delta$, which assigns to each object $\mathcal H$ the corresponding comultiplication $\Delta_\mathcal H$. $\Delta$ is the unique natural transformation from the functor $id$ to $\otimes \colon \textbf{RKHA} \to \textbf{RKHA}$ by $\otimes (\mathcal H) = \mathcal H \otimes \mathcal H$ and $\otimes(T) = T \otimes T$.

\begin{proof}
Let $T \colon \mathcal H_1 \to \mathcal H_2$ be a morphism and observe that
$$\begin{tikzcd}
\mathcal H_1 \arrow[r, "\Delta_{\mathcal H_1}"] \arrow[d, "T"'] & \mathcal H_1 \otimes \mathcal H_1 \arrow[d, "T\otimes T"]\\
\mathcal H_2 \arrow[r, "\Delta_{\mathcal H_2}"'] & \mathcal H_2 \otimes \mathcal H_2
\end{tikzcd}$$
commutes. Therefore, $\Delta$ is a natural transformation between the functors $id$ and $\otimes$.

Suppose that $\Delta'$ was another natural transformation between these two functors. Consider the morphism $T_x \colon \mathbb C \to \mathcal H$ given by $T_x(1) = k_x$. Since the identity is the unique automorphism of $\mathbb C$ as an RKHA on one point and the following diagram commutes,
$$\begin{tikzcd}
\mathbb C \arrow[r, "\Delta'_{\mathbb C}"] \arrow[d, "T_x"] & \mathbb C \otimes \mathbb C \arrow[d, "T_x\otimes T_x"]\\
\mathcal H \arrow[r, "\Delta'_{\mathcal H}"] & \mathcal H \otimes \mathcal H
\end{tikzcd}$$
$\Delta'_{\mathcal H}(k_x) = k_x \otimes k_x$, so $\Delta_\mathcal H ' = \Delta_\mathcal H$.
\end{proof}

The category $\textbf{RKHA}$ can be identified as the largest subcategory of $\textbf{RKHS}$ containing $\mathbb C$, such that there exists a natural transformation between $id$ and $\otimes$.

\begin{defn}
Given an RKHA, $(\mathcal H, \Delta)$, define the cospectrum of $\mathcal H$ as
$$\sigma_{co}(\mathcal H) = \set{\xi \in \mathcal H \backslash \{0\}}{\Delta(\xi) = \xi \otimes \xi}.$$
The cospectrum is equipped with the weak topology from $\mathcal H$.
\end{defn}

By the definition of comultiplication, the cospectrum of $\mathcal H \subset \mathcal L(X)$ contains all nonzero kernel sections $k_x$, $x \in X$. In fact, the cospectrum of $(\mathcal H, \Delta)$ with the weak topology and the spectrum of the abelian Banach algebra $(\mathcal H, \Delta^*, \norm{\cdot}_{Ban})$ with the weak-$*$ topology are isomorphic.

\begin{proof}
Let $\chi \colon \mathcal H \to \mathbb C$ be a continuous linear functional with respect to the Banach algebra norm. Since the Banach algebra norm is equivalent to the Hilbert space norm, $\chi$ is also a bounded linear functional on $\mathcal H$ as a Hilbert space. Then there exists a unique $\xi_{\chi} \in \mathcal H$ such that $\chi(f) = \braket{f}{\xi_\chi}$ and
$$\chi(fg) = \braket{\Delta^*(f \otimes g)}{\xi_\chi} = \braket{f \otimes g}{\Delta(\xi_\chi)}$$
$$\chi(f)\chi(g)=\braket{f}{\xi_\chi}\braket{g}{\xi_\chi} = \braket{f \otimes g}{\xi_\chi\otimes\xi_\chi}$$
Hence, $\chi$ being a character is equivalent to $\xi_\chi \in \sigma_{co}(\mathcal H)$.
By definition, the weak topology on $\sigma_{co}(\mathcal H)$ is the same as the topology on $\sigma(\mathcal H)$.
\end{proof}

Let $\mathcal H \subset \mathcal L(X)$ be a unital RKHA. Then $\braket{1}{k_x} = 1 $ for all $x \in X$, $k_x \neq 0$ and $\Gamma \colon X \to \sigma_{co}(\mathcal H)$ by $\Gamma(x) = k_x$ is well-defined. Similarly, for RKHAs with strictly positive definite kernels $\Gamma \colon X \to \sigma_{co}(\mathcal H)$ is well-defined. In the nonunital case we define $\Gamma \colon X \to \sigma_{co}(\mathcal H) \cup \{0\}$. We refer to these maps collectively as the Gelfand map. Given a morphism between unital RKHAs (resp. nonunital RKHAs) $T \colon \mathcal H_1 \to \mathcal H_2$, then $T$ restricts to a map $T \colon \sigma_{co}(\mathcal H_1) \to \sigma_{co}(\mathcal H_2)$ (resp. $T \colon \sigma_{co}(\mathcal H_1)\cup\{0\} \to \sigma_{co}(\mathcal H_2)\cup\{0\}$) since $\Delta_{\mathcal H_2}T = (T \otimes T)\Delta_{\mathcal H_1}$.

\begin{lem}
Let $\mathcal H \subset \mathcal L(X)$ be an RKHA. Then $\Gamma \colon X \to \sigma_{co}(\mathcal H)\cup\{0\}$ by $\Gamma(x)=k_x$ is continuous iff $\mathcal H \subset C(X)$.
\end{lem}

\begin{proof}
Suppose that $\Gamma$ is continuous and $(x_i)_{i \in I}$ is a net converging to $x$. Then $(k_{x_i})_{i \in I}$ converges to $k_x$ in the weak topology and $f(x_i) = \braket{f}{k_{x_i}}$ which converges to $\braket{f}{k_x} = f(x)$. Since this follows for every net, $\mathcal H \subset C(X)$.

If $\mathcal H \subset C(X)$, then for every $f \in \mathcal H$ and net $(x_i)_{i \in I}$ converging to $x$, $f(x_i) = \braket{f}{k_{x_i}}$ converges to $\braket{f}{k_x} = f(x)$. Therefore $\Gamma$ is continuous.
\end{proof}

Let $\textbf{RKHA}_{\textbf{cont.}}$ denote the subcategory of RKHAs of continuous functions.

\begin{prop}
$\textbf{RKHA}_{\textbf{cont.}}$ is closed under $\otimes$, $\oplus$, pullbacks by continuous maps and pushouts by quotient maps where $\mathcal H_1 \otimes \mathcal H_2 \subset C(X_1 \times X_2)$ and $\mathcal H_1 \oplus \mathcal H_2 \subset C(X_1 \sqcup X_2)$.
\end{prop}

\begin{proof}
    For the tensor product it suffices to show that $\Gamma \colon X_1 \times X_2 \to \sigma_{co}(\mathcal H_1 \otimes \mathcal H_2)\cup\{0\}$ by $\Gamma(x,y)=k_x \otimes k_y$ is continuous. Let $(x_i,y_i)_{i \in I}$ be a net converging to $(x,y)$. Then for $f \in \mathcal H_1$ and $g \in \mathcal H_2$, $\braket{f \otimes g}{k_{x_i} \otimes k_{y_i}} = f(x_i)g(y_i)$ converges to $f(x)g(y)= \braket{f \otimes g}{k_x \otimes k_y}$. Hence, on a dense subspace $ \xi \in span\set{f\otimes g}{f \in \mathcal H_1, g \in \mathcal H_2} \subset \mathcal H_1 \otimes \mathcal H_2$, we have $\braket{\xi}{k_{x_i} \otimes k_{y_i}} \to \braket{\xi}{k_x \otimes k_y}$. By~\eqref{eq:sqrt_kx}, $\set{k_x \otimes k_y}{x \in X_1, y \in X_2}$ is a norm bounded set by $\norm{\Delta}^2$. To show weak convergence of $k_{x_i} \otimes k_{y_i}$ to $k_x \otimes k_y$ and continuity of $\Gamma$, consider $ \varepsilon >0$, $\xi \in \mathcal H_1 \otimes \mathcal H_2$, and choose $\eta \in span\set{f \otimes g}{f \in \mathcal H_1, g \in \mathcal H_2}$ such that $\norm{\xi-\eta} <\varepsilon$. Then
$$\limsup_{i \to \infty}\abs{\braket{\xi}{k_{x_i}\otimes k_{y_i} - k_x \otimes k_y}} \leq 2\norm{\Delta}^2 \varepsilon +\limsup_{i \to \infty} \abs{\braket{\eta}{k_{x_i}\otimes k_{y_i} - k_x \otimes k_y}}\leq 2\norm{\Delta}^2 \varepsilon.$$

Closure under $\oplus$ is trivial.

For the pullback, observe that
$$\begin{tikzcd}
S \arrow[r,"\phi"] \arrow[d,"\Gamma_S"] & X \arrow[d,"\Gamma_X"] \\ \sigma_{co}(\mathcal H(k \circ \phi))\cup\{0\} \arrow[r,"T_\phi"] & \sigma_{co}(\mathcal H(k))\cup\{0\}
\end{tikzcd}$$
commutes and $T_\phi$ is an isometry on $\mathcal H(k \circ \phi)$. Therefore, $\Gamma_S = T_\phi^* \Gamma_X \phi$ is continuous since bounded linear operators are automatically weakly continuous.

For the pushout, observe that
$$\begin{tikzcd}
X \arrow[r,"\phi"] \arrow[d,"\Gamma_X"] & S \arrow[d,"\Gamma_S"] \\ \sigma_{co}(\mathcal H(k))\cup\{0\} \arrow[r,"T_\phi"] & \sigma_{co}(\mathcal H(k_\phi))\cup\{0\}
\end{tikzcd}$$
commutes. Thus, for $U \subset \sigma_{co}(\mathcal H(k_\phi))\cup\{0\}$ open, we have $\phi^{-1}\Gamma_S^{-1}(U) = \Gamma_X^{-1}T_\phi^{-1}(U)$ and $\phi^{-1}\Gamma_S^{-1}(U)$ is open by continuity of $T_\phi\Gamma_X$. Since $\phi$ is a quotient map, $\phi^{-1}\Gamma_S^{-1}(U)$ is open implies that $\Gamma_S^{-1}(U)$ is open, and we conclude that $\Gamma_S$ is continuous.

\end{proof}

Similar to the unitalization of $C^*$-algebras, RKHAs have a well defined unitalization that is compatible with the spectrum.

\begin{defn}
    Let $\mathcal H \subset \mathcal L(X)$ be an RKHA without a unit. Define the unitalization of $\mathcal H$ by $\tilde{\mathcal H}=\mathbb C \oplus \mathcal H \subset \mathcal L(X \cup \{\infty\})$ as a Hilbert space where $1 \oplus 0$ is the constant one function (corresponding to the unit of $\tilde{\mathcal H}$) and $0 \oplus f$ extends $f \in \mathcal H$ by $f(\infty)=0$. Then $\tilde{\mathcal H}$ is a unital RKHA with the kernel $\tilde k \colon (X \cup \{\infty\}) \times (X \cup \{\infty\}) \to \mathbb C$
$$\tilde{k}(x,y)=\left\lbrace\begin{array}{cc} 1+k(x,y) & \text{if } x,y \neq \infty\\ 1 & \text{else} \end{array}\right. .$$
\end{defn}

Note that this definition is different from the direct sum construction $\mathbb C \oplus \mathcal H \subset \mathcal L(X \sqcup \{ \infty \})$ wherein $1 \oplus 0$ is identified with a non-constant function mapping $x \in X$ to 0 and $\infty$ to 1, and thus not acting as a unit under pointwise multiplication. We now verify well-definedness of the unitalization. First, it is immediate that $\tilde k$ is a positive definite function. Second, the inner product on $\tilde{\mathcal H}$ from $\mathbb C \oplus \mathcal H$ coincides with the inner product from the kernel $\tilde{k}$ since $\braket{1 \oplus k_y}{1\oplus k_x} = 1+k(x,y)$. Finally, observe that the new comultiplication is given by
$$\tilde{\Delta} \colon \mathbb C \oplus \mathcal H \to \mathbb C \oplus \mathbb C \otimes \mathcal H \oplus \mathcal H \otimes \mathbb C \oplus \mathcal H \otimes \mathcal H \quad \quad \tilde{\Delta}(\tilde{k}_x)=1 \otimes 1 \oplus 1 \otimes k_x \oplus k_x \otimes 1 \oplus k_x \otimes k_x$$
$$\tilde{\Delta} = \left[ \begin{array}{ccc} 1 & 0\\ 0 & 1 \otimes -\\ 0 & -\otimes 1\\ 0 & \Delta \end{array}\right].$$
which is bounded. Therefore $(\tilde{\mathcal H}, \tilde{\Delta})$ is an RKHA. The definition of $\tilde\Delta$ leads to the following formula for multiplication on $\tilde{\mathcal H}$.
$$ (a1 \oplus f)(b1 \oplus g) = \tilde\Delta^*((a1 \oplus f) \otimes (b1 \oplus g)) = ab1 \oplus ag \oplus bf \oplus fg$$

\begin{prop}
    The spectrum of the unitalization is the one point compactification of the spectrum. Furthermore, the one point compactification of $\sigma(\mathcal H)$ is isomorphic to $\sigma_{co}(\mathcal H)\cup\{0\}$ under the isomorpshism $\sigma_{co}(\mathcal H) \cong \sigma(\mathcal H)$ and identification of $\braket{ \cdot}{1}$ with $0$.
\end{prop}

\begin{proof}
    Define $\Psi \colon \sigma(\mathcal H) \to \sigma(\tilde{\mathcal H})$ by $\Psi(\chi) = \braket{-}{1} +\chi \circ p_{\mathcal H}$ where $p_{\mathcal H}$ is the orthogonal projection onto $\mathcal H$. This is clearly well defined and weak-$*$ to weak-$*$ continuous. In particular, for $\tilde f = a1 \oplus f$ and $\tilde g = b1 \oplus g$,
\begin{multline*}
    \Psi(\chi)(\tilde f \tilde g) = \braket{(a1\oplus f)(b1\oplus g)}{1} + \chi p_\mathcal H ((a1\oplus f)(b1\oplus g))\\
    = ab1 + a\chi(g) + b\chi(f) + \chi(f)\chi(g) = (a + \chi(f))(b+\chi(g)) = (\Psi(\chi)(\tilde f))(\Psi(\chi)(\tilde g)).
\end{multline*}
Observe that $\Psi(\chi)|_\mathcal H =\phi$ is also weak-$*$ continuous from the image of $\Psi$ onto $\sigma(\mathcal H)$. Therefore, $\sigma(\mathcal H)$ is homeomorphic to the image of $\Psi$.

Let $\chi \in \sigma(\tilde{\mathcal H})$ and suppose that $\chi|_{\mathcal H} \neq 0$. Since $\chi$ is unital, $\Psi(\chi|_{\mathcal H}) = \chi$. If $\chi|_{\mathcal H} = 0$ then $\chi = \braket{ \cdot}{1}$. Hence, $\sigma(\tilde{\mathcal H}) = \set{\Psi(\chi)}{\chi \in \sigma(\mathcal H)} \cup \{\braket{\cdot}{1}\} \cong \sigma_{co}(\mathcal H)\cup\{0\}$ is the one point compactification of $\sigma(\mathcal H) \cong \set{\Psi(\chi)}{\chi \in \sigma(\mathcal H)}$.

\end{proof}

The next subcategory of $\textbf{RKHA}$ that we consider are RKHAs with countable weak approximate units, denoted as $\textbf{RKHA}_{\textbf{cwau}}$.

\begin{prop}
$\textbf{RKHA}_{\textbf{cwau}}$ is closed under $\otimes$, $\oplus$, and pullbacks.
\end{prop}

\begin{proof}
    If $\mathcal H_1$ and $\mathcal H_2$ have countable approximate units, then the nets can be paired together in a sum or tensor product, $\eta^{(1)}_n \oplus \eta^{(2)}_n$ and $\eta^{(1)}_n \otimes \eta^{(2)}_n$, which are weak approximate units for $\mathcal H_1 \oplus \mathcal H_2$ and $\mathcal H_1 \otimes \mathcal H_2$. 

Let $\phi \colon S \to X$, $\mathcal H \subset \mathcal L(X)$ and $T_\phi \colon \mathcal H(k \circ \phi) \to \mathcal H$ the corresponding isometry to the pullback construction. The candidate for a countable weak approximate unit is $\left(T_\phi^*(\eta_n)\right)_{n=1}^\infty$. Since $T_\phi: \mathcal H(k \circ \phi) \to \mathcal H$ is an RKHS morphism, $T_\phi^*$ is multiplicative. Moreover, since $T_\phi$ is an isometry, $T_\phi^* T_\phi$ is the identity on $\mathcal H(k \circ \phi)$. Using these facts, for every $f \in \mathcal H$ and $g \in \mathcal H(k \circ \phi)$ we get
$$ M_{T^*_\phi(f)}(g) = T_\phi^*(f)g = T_\phi^*(f)T_\phi^* T_\phi (g) = T_\phi^*(f T_\phi(g))$$
and so $M_{T^*_\phi(f)} = T^*_\phi M_f T_\phi$. Therefore, we have $M_{T^*_\phi(\eta_n)} = T^*_\phi M_{\eta_n} T_\phi$, and weak convergence of $M_{T^*_\phi(\eta_n)}$ to $Id_{\mathcal H(k\circ\phi)}$ follows from weak convergence of $M_{\eta_n}$ to $Id_{\mathcal H}$ and boundedness of $T_\phi$. Boundedness of  $\norm{M_{T_\phi^*(\eta_n)}}_{op}$ follows from norm-boundedness of $M_{\eta_n}$.
\end{proof}

The next theorem concerns the cospectrum and, equivalently, the spectrum as a functor from \textbf{RKHA} to \textbf{Top}. In the unital case, the target of the functor is the category of compact Hausdorff spaces with the cartesian product. In the not necessarily unital case, we use pointed compact Hausdorff spaces with the smash product
$$(X,p) \wedge (Y,q) = X\times Y / \sim \quad \quad (x,q) \sim (p,y) \text{ for all } x \in X, \, y \in Y.$$
We will also need the definition of a monoidal functor which can be found in \cite{EGNO15}.

For unital RKHAs $\mathcal H_1$ and $\mathcal H_2$ define the map $ \Phi \colon \sigma_{co}(\mathcal H_1) \times \sigma_{co}(\mathcal H_2) \to \sigma_{co}(\mathcal H_1 \otimes \mathcal H_2)$ by $\Phi(\xi_1,\xi_2) = \xi_1 \otimes \xi_2$. In the not necessarily unital case, we similarly consider $\Phi \colon \sigma_{co}(\mathcal H_1) \cup\{0\} \times \sigma_{co}(\mathcal H_2)\cup\{0\} \to \sigma_{co}(\mathcal H_1 \otimes \mathcal H_2)\cup \{0\}$ with $\Phi(\xi_1,\xi_2) = \xi_1 \otimes \xi_2$ and define $\tilde{\Phi} \colon \sigma_{co}(\mathcal H_1) \cup \{ 0 \} \times \sigma_{co}(\mathcal H_2) \cup \{ 0 \}/\Phi^{-1}(0) \to \sigma_{co}(\mathcal H_1 \otimes \mathcal H_2) \cup \{ 0 \}$ by identifying $\Phi^{-1}(0)$ as a single point.

\begin{thm}
    \label{thm:spec}
    For unital RKHAs, the map $\Phi$ is a natural isomorphism that, together with the spectrum, $sp(\mathcal H) = \sigma(\mathcal H) \cong \sigma_{co}(\mathcal H)$, defines a monoidal functor $(sp,\Phi) \colon (\textbf{RKHA}_\textbf{u},\otimes) \to (\textbf{Top}_\textbf{cpt,Haus}, \times)$. On the morphism spaces $T \in Mor(\mathcal H_1, \mathcal H_2)$, the functor is defined by $sp(T) = T|_{\sigma_{co}(\mathcal H_1)}$. Similarly, for RKHAs with countable weak approximate units, $(\tilde{sp},\tilde{\Phi}) \colon (\textbf{RKHA}_{\textbf{cwau}},\otimes) \to (\textbf{Top}_{\textbf{cpt,Haus}}^*,\wedge )$ defines a monoidal functor given by $\tilde{sp}(\mathcal H) = (\sigma(\mathcal H)\cup \{ \braket{-}{0} \}, 0 ) \cong (\sigma_{co}(\mathcal H) \cup\{0\},0)$ and $\tilde{sp}(T) = T|_{\sigma_{co}(\mathcal H_1)\cup\{0\}}$ where $\textbf{Top}^*_{\textbf{cpt,Haus}}$ is equipped with the smash product $\wedge$.
\end{thm}

\begin{proof}
    Starting from the unital case, we have already observed that $T|_{\sigma_{co}(\mathcal H_1)} \colon \sigma_{co}(\mathcal H_1) \to \sigma_{co}(\mathcal H_2)$ for unital RKHAs $\mathcal H_1$ and $\mathcal H_2$. This map is also weak to weak continuous since $T$ is a bounded linear map. Clearly $sp(T_1 \circ T_2) = sp(T_1) \circ sp(T_2)$ and $sp(id_\mathcal H) = id_{sp(\mathcal H)}$. Thus $sp$ is a functor from $\textbf{RKHA}_\textbf{u}$ to $\textbf{Top}_\textbf{cpt,Haus}$.

    Next, to show monoidal functoriality of $(sp, \Phi)$, we must show that $\Phi$ is a homeomorphism and a natural transformation from the functor $\times \circ (sp,sp)$ to $sp \circ (- \otimes -)$. Since the cospectrum of an RKHA is a norm bounded set, any net $(\xi^1_i,\xi^2_i)_{i \in I} \subset sp(\mathcal H_1) \times sp(\mathcal H_2)$ converging to $(\xi^1,\xi^2)$ in weak$\times$weak has norm bounded components. Since
$$\lim_{i\to\infty} \braket{f \otimes g}{\Phi(\xi^1_i,\xi^2_i)} = \braket{f}{\xi^1}\braket{g}{\xi^2},$$ and the net $\left(\Phi(\xi^1_i,\xi^2_i)\right)_{i \in I}$ is norm bounded we have convergence to $\Phi(\xi^1,\xi^2)$ in the weak topology. Therefore $\Phi$ is weak$\times$weak to weak continuous.
Injectivity of $\Phi$ follows from $\braket{f \otimes 1}{\xi_1 \otimes \xi_2} =\braket{f}{\xi^1}$ and $\braket{1 \otimes f}{\xi_1 \otimes \xi_2} =\braket{f}{\xi_2}$. For surjectivity, it is easier to work with the space of characters instead of the cospectrum. It suffices to show that every character of $\mathcal H_1 \otimes \mathcal H_2$ is the tensor product of two characters $\chi_1$ and $\chi_2$ of $\mathcal H_1$ and $\mathcal H_2$ respectively. Observe that $\mathcal H_1 \cong \mathcal H_1 \otimes 1 \subset \mathcal H_1 \otimes \mathcal H_2 \supset 1 \otimes \mathcal H_2 \cong \mathcal H_2$, and for $\chi \in \sigma(\mathcal H_1 \otimes \mathcal H_2)$, $\chi(f \otimes g) = \chi(f \otimes 1)\chi(1 \otimes g)$. Hence $\chi$ is determined by its restriction to $\mathcal H_1$ and $\mathcal H_2$, $\chi = \chi|_{\mathcal H_1 \otimes 1} \otimes \chi|_{1 \otimes \mathcal H_2}$. Since the spectra are compact and Hausdorff, $\Phi$ is a homeomorphism.

We now show that $\Phi$ is a natural transformation between the functors $\times \circ (sp,sp)$ and $sp \circ (- \otimes -)$. Let $T \colon \mathcal H_1 \to \mathcal H_2$ and $S \colon \mathcal K_1 \to \mathcal K_2$ be RKHA morphisms. Then $sp(T \otimes S) = T \otimes S|_{\sigma_{co}(\mathcal H_1) \otimes \sigma_{co}(\mathcal K_1)} = sp(T) \times sp(S)$ and the following diagram commutes
$$\begin{tikzcd}
sp(\mathcal H_1) \times sp(\mathcal K_1) \arrow[r, "\Phi_{\mathcal H_1, \mathcal K_1}"] \arrow[d,"sp(T)\times sp(S)"]  & sp(\mathcal H_1 \otimes \mathcal K_1) \arrow[d,"sp(T \otimes S)"]\\
sp(\mathcal H_2) \times sp(\mathcal K_2) \arrow[r, "\Phi_{\mathcal H_1, \mathcal K_1}"]  & sp(\mathcal H_2 \otimes \mathcal K_2)
\end{tikzcd}.$$
The monoidal structure axiom is also easily verified (see diagram 2.23 in \cite{EGNO15}). Finally, $sp(\mathbb C) = \{1\}$ which is the identity in the monoidal category \textbf{Top}.

In the countable weak approximate unit case, functoriality of $\tilde{sp}$ is proven identically to the unital case with the observation that $T|_{\sigma_{co}(\mathcal H) \cup\{0\}}(0)=0$ making it a morphism between pointed topological spaces. Just as before the cospectrum is a norm bounded set. This implies $\Phi \colon \sigma_{co}(\mathcal H_1) \cup\{0\} \times \sigma_{co}(\mathcal H_2)\cup\{0\} \to \sigma_{co}(\mathcal H_1 \otimes \mathcal H_2)\cup \{0\}$ with $\Phi(\xi_1,\xi_2) = \xi_1 \otimes \xi_2$ is continuous. We first show $\Phi$ is surjective. Similarly, it is easier to work with characters instead of the cospectrum. Let $\chi \in \sigma( \mathcal H_1 \otimes \mathcal H_2)$ and $(\eta_i^1)_{i \in \mathbb N}$, $(\eta_i^2)_{i \in \mathbb N}$ be countable weak approximate units for $\mathcal H_1$ and $\mathcal H_2$. Then by weak continuity of $\chi$
$$\chi(f \otimes g) = \lim_{i\to \infty} \lim_{j\to \infty} \chi(f \eta_i^1 \otimes \eta^2_j g) = \left(\lim_{i\to\infty}\chi(\eta_i^1 \otimes g)\right)\left(\lim_{j\to\infty} \chi(f \otimes \eta^2_j)\right)$$
which implies convergence of the following definitions
$$\chi|_{\mathcal H_1}(f) = \lim_{j\to\infty} \chi(f \otimes \eta_j^2)$$
$$\chi|_{\mathcal H_2}(g) = \lim_{i\to\infty} \chi(\eta_i^1 \otimes g).$$
These are both characters since
$$\chi|_{\mathcal H_1}(fg) = \lim_{i\to\infty} \lim_{j\to\infty} \chi(fg \otimes \eta_i^2 \eta_j^2) = \left(\lim_{i\to\infty}\chi(f \otimes \eta_i^2)\right)\left(\lim_{j\to\infty} \chi(g \otimes \eta^2_j)\right) = \chi|_{\mathcal H_1}(f) \chi|_{\mathcal H_1}(g)$$
and $\chi(f \otimes g) = \chi|_{\mathcal H_1}(f) \chi|_{\mathcal H_2}(g)$ implies boundedness. Hence $\chi = \chi|_{\mathcal H_1} \otimes \chi|_{\mathcal H_2}$ and $ \Phi$ is onto. $\Phi$ is also injective when restricted to $\sigma_{co}(\mathcal H_1) \times \sigma_{co}(\mathcal H_2)$. Suppose that $\xi_1 \otimes \xi_2 = \zeta_1 \otimes \zeta_2 \neq 0$ for $\xi_i, \zeta_i \in \sigma_{co}(\mathcal H_i)$. Then
$$\braket{\xi_1}{\zeta_1}\braket{\xi_2}{\zeta_2}=\braket{\xi_1 \otimes \xi_2}{\zeta_1 \otimes \zeta_2} = \braket{\xi_1 \otimes \xi_2}{\xi_1 \otimes \xi_2} = \braket{\xi_1}{\xi_1}\braket{\xi_2}{\xi_2} \neq 0$$
which implies that $\xi_1 = z \zeta_1$ and $\xi_2 =\frac{1}{z} \zeta_2$ for some $z \in \mathbb C \backslash{0}$. Since $\xi_1,\zeta_1 \in \sigma_{co}(\mathcal H_1)$,
$$z^2 \zeta_1 \otimes \zeta_1 = \xi_1 \otimes \xi_1 = \Delta_{\mathcal H_1}(\xi_1) = z \Delta_{\mathcal H_1}(\zeta_1) = z \zeta_1 \otimes \zeta_1$$
which implies that $z=1$.
$\Phi$, however is not injective on $\tilde{sp} \times \tilde{sp}$ since $\Phi(0,\chi) = 0$. By quotienting out by $\Phi^{-1}(0)$, $ \tilde{\Phi} \colon \tilde{sp}(\mathcal H_1) \times \tilde{sp}(\mathcal H_2)/\Phi^{-1}(0) \to \tilde{sp}(\mathcal H_1 \otimes \mathcal H_2)$ is a continuous bijection with the smash product. Since $\tilde{\Phi}$ is a continuous bijection between compact Hausdorff spaces, it is an isomorphism between the product in $\textbf{Top}_{\textbf{cpt,Haus}}^*$ and $\tilde{sp}(\mathcal H_1 \otimes \mathcal H_2)$. Naturality of $\tilde{\Phi}$ between the functors $\times \circ (\tilde{sp},\tilde{sp})$ and $\tilde{sp} \circ (- \otimes -)$ follows by a the same argument as the unital case with a quotient.
\end{proof}

\subsection{The GRS condition}

So far we have a category of RKHAs which is closed under various constructions but we do not know if these lead to non-isomorphic Banach algebras. We will use the GRS and BD conditions to show that \textbf{RKHA} has at least one object for every compact subset of $\mathbb R^n$, $n>0$.

Let us return to the example of RKHAs from compact abelian groups, $\mathcal H_\lambda$, and in particular $G = \mathbb T^n$, $\lambda \colon \mathbb Z^n \to \mathbb R_{>0}$, and the subconvolutive weight $\lambda^{-1}(k) =e^{\tau \norm{k}_p^p}$, $0< \tau$, $0 <p <1$ from example~\ref{subexponential weights}(i). As shown in \cite{DG23,DGM23},
$\sigma(\mathcal H_\lambda) \cong G$ via the Gelfand map since this weight satisfies the GRS condition. We now extend this to some locally compact groups. Similar results are already known for compactly generated groups of polynomial growth and convolution algebras, $L_\omega^1(\hat{G})$ (see \cite{FGL06}). In the abelian case, compactly generated groups of polynomial growth are classified by two natural numbers, $d$, $e$, and a compact abelian group $K$, $\hat{G} \cong \mathbb R^d \times \mathbb Z^e \times K$ (see \cite{HR79}). The authors of \cite{FGL06} used submultiplicativity of the weight to apply the GRS condition to components of this product. Since the weights we consider for nonunital RKHAs may not be submultiplicative, we cannot reduce the application of the GRS condition to the individual components of this decomposition. Hence we will drop the compact group $K$ from our analysis in the following theorem.

\begin{thm}
    \label{thm:GRS}
    Let $\hat{G} \cong \mathbb R^d \times \mathbb Z^e$ and $\lambda \in L^1(\hat{G}) \cap C_0(\hat{G})$ be subconvolutive, symmetric ($\lambda(-\alpha) = \lambda(\alpha)$), and strictly positive. If $\lambda$ satisfies the GRS condition~\eqref{eq:GRS} then $\sigma(\mathcal H_\lambda) \cong G$ by the Gelfand map $\Gamma: G \to \sigma_{co}(\mathcal H_\lambda) \cong \sigma(\mathcal H_\lambda)$.
\end{thm}

\begin{proof}
    First, note that $\Gamma : G \to \sigma_{co}(\mathcal H_\lambda)$ is well-defined even if $\mathcal H_\lambda$ is nonunital by strict positivity of $\lambda$. Next, let $\chi = \hat{\mathcal F}(\hat \chi) \in \sigma_{co}(\mathcal H_\lambda)$ with $\hat\chi \in L^1(\hat G)$ and, using~\eqref{eq:T_nonunital}, observe that
$$\iint_{\hat{G}\times \hat{G}} \frac{\hat{\chi}(\alpha+\beta)}{\lambda(\alpha+\beta)} \sqrt{\lambda(\alpha)\lambda(\beta)} \psi_\alpha(x)\psi_\beta(y)d\hat{\mu} \times \hat{\mu}(\alpha,\beta)=\Delta(\chi)(x,y)$$
$$= \chi(x)\chi(y) = \iint_{\hat{G}\times \hat{G}} \frac{\hat{\chi}(\alpha)\hat{\chi}(\beta)}{\lambda(\alpha)\lambda(\beta)}  \sqrt{\lambda(\alpha)\lambda(\beta)} \psi_\alpha(x)\psi_\beta(y)d\hat{\mu} \times \hat{\mu}(\alpha,\beta).$$
Define $\phi=\frac{\hat{\chi}}{\lambda} \neq 0$. By injectivity of $\hat{\mathcal F} \colon L^1(\hat{G} \times \hat{G}) \to C_0(G\times G)$, locally for almost every $(\alpha,\beta) \in \hat{G}\times \hat{G}$,
$\phi(\beta-\alpha)=\phi(-\alpha)\phi(\beta)$. For $K \subset \hat{G}$ compact and $E = \set{(\alpha,\beta)}{\phi(\beta-\alpha)\neq \phi(-\alpha)\phi(\beta)}$, we can apply Fubini-Tonelli to the indicator function $1_{E \cap K\times K}$. This implies that locally for almost every $\alpha \in \hat{G}$, $\phi(\beta-\alpha)=\phi(-\alpha)\phi(\beta)$ as functions of $\beta$ in $L^1_{loc}(\hat{G})$ (i.e. the space of locally integrable functions with seminorms $\int_{U_k} \abs{\cdot} d\hat{\mu}$ for $U_k \subset U_{k+1}$ compactly included and $\hat{G} = \bigcup_k U_k$).

We will show that $\phi$ is equal locally almost everywhere to a continuous group homomorphism to $\mathbb C$.  Let $S_\alpha(\phi) (\beta) = \phi(\beta-\alpha)$, and observe that $\alpha \to S_\alpha \phi$ is continuous in the $L^1_{loc}(\hat{G})$ topology (since $\hat{G} \cong \mathbb R^d \times \mathbb Z^e$, approximate $\phi$ locally by a $C^\infty_c(\mathbb R^d \times \mathbb Z^e)$ function to prove continuity in every seminorm). Let $\phi = \psi \abs{\phi}$ be the polar decomposition of $\phi$ so that $\psi \colon \hat{G} \to \mathbb T^1 \cup{\{0\}}$. Since $\phi$ is nonzero there exists a compact non-null set $E \subset \hat{G}$ with $\int_E \abs{\phi} d\hat{\mu} \neq 0$. Since $f \mapsto \int_E f \overline{\psi} d\hat{\mu}$ is continuous from $L^1_{loc}(\hat{G})$ to $\mathbb C$, the map $\tilde{\phi} \colon \hat{G} \to \mathbb C$
$$\tilde{\phi}(\alpha) = \frac{1}{\int_E \abs{\phi} d\hat{\mu}} \int_E S_{-\alpha}(\phi) \overline{\psi} d\hat{\mu}$$
is continuous. Furthermore, locally for almost every $\alpha \in \hat{G}$, $\tilde{\phi}(\alpha) = \phi(\alpha)$. Then locally for almost every $\alpha \in \hat{G}$, $S_\alpha(\tilde{\phi}) = S_\alpha(\phi) = \phi(-\alpha)\phi = \tilde{\phi}(-\alpha)\tilde{\phi}$ as functions in $L^1_{loc}(\hat{G})$. By continuity we have $\tilde{\phi}(\alpha+\beta) = \tilde{\phi}(\alpha)\tilde{\phi}(\beta)$ for every $\alpha, \beta \in \hat{G}$. As a consequence of $\tilde{\phi} \neq 0$, $\tilde{\phi}(\alpha) \neq 0$ for any $\alpha\in\hat G$ and $\tilde{\phi}(-\alpha) = \tilde{\phi}(\alpha)^{-1}$.

We may choose the $L^1_{loc}$ and $L^1$ representatives, $\phi = \tilde{\phi}$ and $\hat{\chi}=\lambda \tilde{\phi}$ which are continuous. Finally, we will show by contradiction that $\phi$ is unimodular using the GRS condition. Suppose that $\abs{\phi(\gamma)} = r >1$. Fix a compact neighborhood $F$ of $\gamma$ with $\inf_{\alpha \in F} \abs{\phi(\alpha)} = \rho >1$ and $\sup_{\alpha \in F} \abs{\phi(\alpha)} = M$. Then $\rho^k \leq \abs{\phi(\alpha)} \leq M^k$ for $\alpha \in kF =\set{k\cdot \alpha}{\alpha \in F}$ and there exists an increasing sequence $\{k_i\}_{i=1}^\infty$ such that $\{k_iF\}_{i=1}^\infty$ are disjoint. For $\alpha \in F$, the GRS condition gives
$$\liminf_{k \to \infty} \abs{\hat{\chi}(k\alpha)}^{1/k} = \liminf_{k \to \infty} \lambda(k\alpha)^{1/k} \abs{\phi(\alpha)} \geq \rho >1.$$
Therefore there is an $n_\alpha$ such that $k \geq n_\alpha$ implies that $\abs{\hat{\chi}(k\alpha)} \geq 1$. Define $f_n \colon F \to \mathbb R_{>0}$ by $f_n(\alpha) = \inf_{k\geq n} \abs{\hat{\chi}(k\alpha)}$ and $F_n = f_n^{-1}([1,\infty))$. Then due to the GRS condition $F =\bigcup_{n=1}^\infty F_n$. Since $\hat{\chi}$ is continuous, $f_n$ is an increasing sequence of upper semicontinuous functions, $F_n$ is measurable and $F_{n-1} \subset F_n$. Furthermore, for $k \geq n$ and $\alpha \in kF_n$, $\abs{\hat{\chi}(\alpha)} \geq 1$. Therefore, there exists an $n_0$ such that $\hat{\mu}(F_{n_0}) \neq 0$ and an $i_0$ such that $k_i \geq n_0$ for $i \geq i_0$ and so
$$\sum_{i=i_0}^\infty \hat{\mu}(k_iF_{n_0}) \leq \sum_{i=i_0}^\infty \int_{k_iF_{n_0}} \abs{\hat{\chi}}d\hat{\mu} \leq \norm{\hat{\chi}}_{L^1(\hat{G})} < \infty.$$

Since $\hat{\mu}(kF) \geq \hat{\mu}(F)$ for $\hat{G} \cong \mathbb R^d \times \mathbb Z^e$, $\hat{\chi}$ cannot belong to $L^1(\hat{G})$. This is a contradiction and so our assumption that $\phi$ was not unimodular is false. Hence, the Gelfand map $\Gamma \colon G\to \sigma_{co}(\mathcal H_\lambda)$, $\Gamma(x) = k_x$ is onto. Injectivity and continuity follow from injectivity of $\hat{\mathcal F}$ and $\mathcal H_\lambda \subset C_0(G)$.

Finally, we must show $\Gamma^{-1}$ is continuous. Given a weak-$*$ convergent net $(k_{x_i})_{i \in I} \to k_x$ in $\sigma_{co}(\mathcal H_\lambda)$, we know that $\braket{f}{k_{x_i}} \to f(x)$. It suffices to show that for every open set $ x \in U \subset G \cong \mathbb R^d \times \mathbb T^e$ there is a function $f \in \mathcal H_\lambda$ such that $\abs{f(x)} > sup_{y \notin U} \abs{f(y)}$ as this forces $x_i$ to eventually be in $U$. Such a function is given by $f_n = \left(\frac{k_x}{k(x,x)}\right)^n \in \mathcal H_\lambda$ for $n$ sufficiently large. This follows from the fact that $k_x$ attains its supremum only at $x$ (apply the Cauchy-Schwarz inequality and injectivity of $\hat{\mathcal F}$ to $k_x$ and $k_y$).
\end{proof}

The weights given in example \ref{subexponential weights} also satisfy the BD condition~\eqref{eq:BD} (which implies the GRS condition)
This condition characterizes existence of test functions with compact support (see \cite{Grochenig07}). For an explicit construction of functions with compact support in arbitrarily small open sets and Fourier coefficients with decay faster than $\lambda^{-1}(k)$, see \cite{Ing34}. The BD condition only applies to RKHAs built from locally compact abelian groups. The restriction property below applies to all RKHAs and captures the notion of having compactly supported functions.

Since an RKHA $\mathcal H \subset \mathcal L(X)$ may live on a space $X$ without a topology, we must consider $\mathcal H$ as a space of continuous functions $\mathcal H \subset C(\sigma_{co}(\mathcal H))$ by $f(x) :=\braket{f}{x}$ for $f \in \mathcal H$ and $x \in \sigma_{co}(\mathcal H)$. Since $\sigma_{co}(\mathcal H) \cong \sigma(\mathcal H)$, this is equivalent to the identification $\mathcal H \subset C(\sigma(\mathcal H))$ made earlier. Moreover, $\sigma_{co}(\mathcal H)$ trivially acquires the feature map $\phi: \sigma_{co}(\mathcal H) \to \mathcal H$ by $\phi = id\rvert_{\sigma_{co}(\mathcal H)}$, so we can view $\mathcal H$ as an RKHS on $\sigma_{co}(\mathcal H)$ with reproducing kernel $(x, y) \mapsto \braket{y}{x}$ for $x,y \in \sigma_{co}(\mathcal H)$. We will still denote the sections of this kernel by $k_x$. For $Y \subset \sigma_{co}(\mathcal H)$, let $\iota \colon Y \to \sigma_{co}(\mathcal H)$ denote the inclusion map and for brevity let $\mathcal H(Y) = \mathcal H(k \circ \iota)$ denote the pullback and $T_\iota \colon \mathcal H(Y) \to \mathcal H$ the corresponding isometry. Since $\Delta_{\mathcal H}(T_\iota((k\circ \iota)_y)) = k_{\iota(y)} \otimes k_{\iota(y)}$, $\Delta_{\mathcal H}(T_\iota(\mathcal H(Y))) \subset T_\iota(\mathcal H(Y)) \otimes T_\iota(\mathcal H(Y))$, and we have a simple formula for the cospectrum of a pullback RKHA
$$T_\iota(\sigma_{co}(\mathcal H(Y))) = \set{T_\iota(\xi) \in \mathcal H(Y) \backslash \{0\}}{\Delta_{\mathcal H(Y)}(\xi) = \xi \otimes \xi}$$
$$=\set{T_\iota(\xi) \in \mathcal H(Y) \backslash \{0\}}{\Delta_{\mathcal H}(T_\iota(\xi)) = T_{\iota}(\xi) \otimes T_\iota(\xi)} = \sigma_{co}(\mathcal H) \cap T_{\iota}(\mathcal H(Y)).$$

\begin{defn}
    Let $\mathcal H \subset C(\sigma_{co}(\mathcal H))$ be an RKHA. $\mathcal H$ has the \textbf{compact support property} if for every point $x \in \sigma_{co}(\mathcal H)$ and open neighborhood $U \subset \sigma_{co}(\mathcal H)$ containing $x$ there exists a function $f \in \mathcal H$ such that $f(x) \neq 0$ and $f|_{{\sigma_{co}(\mathcal H)}\backslash U} = 0$. $\mathcal H$ has the \textbf{restriction property} if for all $Y \subset \sigma_{co}(\mathcal H)$ closed, $T_\iota(\sigma_{co}(\mathcal H(Y))) = Y$.
\end{defn}

Since $\sigma_{co}(\mathcal H)$ is a locally compact Hausdorff space in the weak topology, the compact support property implies that for every compact set $S \subseteq \sigma_{co}(\mathcal H)$ with nonempty interior there exists a function $f\in \mathcal H$ whose support is contained in $S$. In fact, the compact support and restriction properties are equivalent.

\begin{proof}
Let $\mathcal H$ be an RKHA with the restriction property and $x \in U \subset \sigma_{co}(\mathcal H)$ a point and open neighborhood. Let $Y = \sigma_{co}(\mathcal H)\backslash U$ with inclusion map $\iota \colon Y \to \sigma_{co}(\mathcal H)$ and corresponding isometry $T_\iota \colon \mathcal H(Y) \to \mathcal H$. Since $T_\iota(\sigma_{co}(\mathcal H(Y))) = \sigma_{co}(\mathcal H) \cap T_\iota(\mathcal H(Y)) = Y$, $x \notin T_\iota(\mathcal H(Y))$, and $T_\iota(\mathcal H(Y)) = \set{f\in \mathcal H}{f|_Y = 0}^\perp$, then there must be a function in $\mathcal H$ that vanishes on $Y$ and not at $x$.

Let $\mathcal H$ be an RKHA with the compact support property. Let $Y \subset \sigma_{co}(\mathcal H)$ be a closed subset. Then for every $x \in \sigma_{co}(\mathcal H) \backslash Y$, since there is a function $f \in \mathcal H$ that vanishes on $Y$ and is nonzero at $x$, $\braket{f}{x} \neq 0$, and so $x \notin \set{f\in \mathcal H}{f|_Y = 0}^\perp = \mathcal H(Y)$. Therefore $x \notin T_\iota(\sigma_{co}(\mathcal H(Y))) = \sigma_{co}(\mathcal H) \cap T_\iota(\mathcal H(Y))$ and $T_\iota(\sigma_{co}(\mathcal H(Y))) = Y$.
\end{proof}

\begin{prop}
Let $\mathcal H_\lambda$ be an RKHA on a locally compact abelian group. Then it has the compact support property iff $\lambda$ satisfies the BD condition \eqref{eq:BD}.
\end{prop}

\begin{proof}
Observe that $f \in \mathcal H_\lambda$ iff $f *\overline{f} \in \hat{\mathcal F} L^1_{1/\lambda}(\hat{G})$ and $f,g \in \mathcal H_\lambda$ implies that $f * g \in \hat{\mathcal F} L^1_{1/\lambda}(\hat{G})$. Suppose $\mathcal H_\lambda$ has the compact support property and fix an open set $E \subset G$. Since $G$ is a topological group we can pick a small enough support for $f \in \mathcal H_\lambda \backslash \{0\}$ such that $supp(f * \overline{f}) \subset E$. By \cite{Domar56}, $\lambda$ must satisfy the BD condition.

Suppose that $\lambda$ satisfies the BD condition. Then $\sqrt{\lambda}$ satisfies the BD condition and by Young's convolution identity $\norm{\hat{f} * \hat{g}}_{L^2_{1/\sqrt{\lambda}}(\hat{G})} \leq \norm{f}_{L^1_{1/\sqrt{\lambda}}(\hat{G})} \norm{g}_{L^2_{1/\sqrt{\lambda}}(\hat{G})}$ for $f \in \hat{\mathcal F}L^1_{1/\sqrt{\lambda}}(\hat{G})$ and $g \in \hat{\mathcal F}L^2_{1/\sqrt{\lambda}}(\hat{G})$. For $x \in U \subset G$, choose $f$ and $g$ to be nonzero at $x$ and $\text{supp}(f) \subset U$. Therefore $fg \in \mathcal H_\lambda$, $\text{supp}(fg) \subset U$, and $f(x)g(x) \neq 0$.
\end{proof}

Let $\textbf{RKHA}_{\textbf{cs}}$ denote the subcategory of RKHAs with the compact support property.
\begin{prop}
$\textbf{RKHA}_{\textbf{cs}}$ is closed under $\otimes$, $\oplus$, pullbacks, and pushouts.
\end{prop}

\begin{proof}
Closure under $\otimes$ and $\oplus$ is immediate.

Let $\phi \colon S \to X$, $T_\phi \colon \mathcal H(k \circ \phi) \to \mathcal H(k)$ be a pullback. Since $T_\phi$ is an isometry, its restriction $T_\phi \colon \sigma_{co}(\mathcal H(k \circ \phi)) \to \sigma_{co}(\mathcal H(k))$, and partial inverse $T_\phi^* \colon Im(T_\phi) \to \sigma_{co}(\mathcal H(k \circ \phi))$ are weak-$*$ continuous. Hence, $T_\phi$ defines an inclusion map and $\mathcal H(k \circ \phi) \subset \mathcal L(\sigma_{co}(\mathcal H(k \circ \phi)))$ is obtained from $\mathcal H \subset \mathcal L(\sigma_{co}(\mathcal H))$ by the restriction. The restriction RKHA clearly satisfies the compact support property.

Let $\phi \colon X \to S$, $T_\phi \colon \mathcal H(k) \to \mathcal H(k_\phi)$ be a pushout. Let $Y \subset \sigma_{co}(\mathcal H(k_\phi))$ be closed, then $T_\phi^{-1}(Y) \subset \sigma_{co}(\mathcal H(k))$ is closed. The following diagram commutes
$$\begin{tikzcd}
T_\phi^{-1}(Y) \arrow[r,"T_\phi|_{T_\phi^{-1}(Y)}"] \arrow[d, "\Gamma_{T_\phi^{-1}(Y)}"] & Y \arrow[d,"\Gamma_Y"] \\ \sigma_{co}(\mathcal H(T^{-1}_\phi(Y))) \arrow[r,"T_\phi"] & \sigma_{co}(\mathcal H(k_\phi)(Y))
\end{tikzcd}.$$
Since $T_\phi$ is onto, its restriction to $T_\phi \colon \sigma_{co}(\mathcal H(T^{-1}_\phi(Y))) \to \sigma_{co}(\mathcal H(k_\phi)(Y))$ is also onto. $\mathcal H$ satisfies the Gelfand property by assumption, and so $\Gamma_{T^{-1}_\phi(Y)}$ is onto. This forces $\Gamma_Y$ to be onto. Since $Y \subset \sigma_{co}(\mathcal H(k_\phi))$, $\Gamma_Y$ is a homeomorphism between $Y$ and its image, hence $\sigma(\mathcal H(k_\phi)(Y)) \cong Y$.
\end{proof}

Observe that the weights in example \ref{subexponential weights} satisfy the BD condition, hence the GRS condition. Combining these examples with the many constructions above, \textbf{RKHA} contains objects $\mathcal H$ with $\sigma(\mathcal H)$ isomorphic to any locally compact subset of $\mathbb R^n$.

\subsection{Banach algebra quotients of RKHAs}
We have focused on constructions of RKHAs stemming from the RKHS category. We now address Banach algebra constructions and their compatibility with RKHAs. First, subalgebras of RKHAs yield new RKHAs. Let $\mathcal M \subset \mathcal H \subset \mathcal L(X)$ be a subalgebra of an RKHA. Then $\overline{\mathcal M}^{\norm{\cdot}} \subset \mathcal L(X)$ is an RKHA.

\begin{proof}
Let $(f_n)_{n=1}^\infty, (g_n)_{n=1}^\infty \subset \mathcal M$ be sequences converging to $f,g \in \overline{\mathcal M}^{\norm{\cdot}}$. Then clearly, $\Delta^*(f_n \otimes g_n)$ converges to $fg \in \overline{\mathcal M}^{\norm{\cdot}}$. Therefore $\Delta^*|_{\overline{\mathcal M} \otimes \overline{\mathcal M}} \colon \overline{\mathcal M} \otimes \overline{\mathcal M} \to \overline{\mathcal M}$ is bounded and implements point-wise multiplication.
\end{proof}

The next construction from Banach algebras we consider is the quotient Banach algebra. Let $\mathcal H \subset \mathcal L(X)$ be a unital RKHA and consider the closed ideal $I_Y = \set{f \in \mathcal H}{f|_Y = 0} = \mathcal H(Y)^\perp$ for a subset, $Y$, of $X$. This provides two different quotients of $\mathcal H$ as a Hilbert space and a Banach algebra. As a Hilbert space $\mathcal H(Y) \cong \mathcal H/I_Y \subset \mathcal L(Y)$ by $\mathcal H/I_Y \ni [f] \mapsto g \in \mathcal H(Y)$ where $g$ is the unique element of $[f]$ such that $\norm{g}_\mathcal H = \inf_{h \in [f]} \norm{h}_\mathcal H$ and for all $g \in [f]$, $g|_Y = f|_Y$. We may also build the quotient Banach algebra $B= \mathcal H/I_Y$ where $\norm{[f]}_B = \inf_{g \in [f]} \norm{M_g}_{op}$. Since $\mathcal H(Y)$ is a unital RKHA it has a Banach algebra norm which may be different than the norm for $B$. However, these Banach algebra norms will be equivalent since
$$\frac{1}{C}\norm{f}_{\mathcal H(Y)} = \frac{1}{C} \inf_{g \in [f]} \norm{g}_\mathcal H \leq \inf_{g \in [f]} \norm{M_g}_{op} \leq C \inf_{g \in [f]} \norm{g}_\mathcal H =C\norm{f}_{\mathcal H(Y)} \quad \forall f \in \mathcal H(Y).$$
Hence, the quotient Banach algebra and restriction RKHA generate Banach algebras with equivalent norms.

Finally, we address the metric topology induced by $\mathcal H$ on the spectrum $\sigma(\mathcal H) \cong \sigma_{co}(\mathcal H)$. Define the metric $d \colon \sigma(\mathcal H) \times \sigma(\mathcal H) \to \mathbb R$ by $d(x,y) = \norm{k_x - k_y}$. In general, the metric topology and weak-$*$ topology may not agree on $\sigma(\mathcal H)$.
\begin{prop}\label{metric topology}
Let $\mathcal H$ be an RKHA. The metric topology and weak-$*$ topology on $\sigma(\mathcal H)$ agree iff the unit ball of $\mathcal H$ is equicontinuous with respect to the weak-$*$ topology on $\sigma(\mathcal H)$. These are both equivalent to continuity of $\kappa(x) =k(x,x)$ in the weak-$*$ topology.
\end{prop}

\begin{proof}
If the metric topology and weak-$*$ topology agree then for all $f \in (\mathcal H)_1$
$$\abs{f(x)-f(y)} = \abs{\braket{f}{k_x - k_y}} \leq d(x,y).$$
Fix $y \in \sigma(\mathcal H)$, $\varepsilon >0$, and pick $U=\set{x\in \sigma(\mathcal H)}{d(x,y) <\varepsilon}$. Then clearly $\abs{f(x)-f(y)} < \varepsilon$.

Now suppose that $(\mathcal H)_1$ is equicontinuous. Since the kernel functions of an RKHA are uniformly norm bounded, they also form a equicontinuous family. Let $(x_i)_{i \in I}$ be a net in $\sigma(\mathcal H)$ converging to $x$ in the weak-$*$ topology. For $\varepsilon >0$ pick a weak-$*$ neighborhood $U$ of $x$ such that $\abs{k_y(x)-k_y(x_i)} < \varepsilon$ for all $x_i \in U$. In particular, $\abs{k_{x_i}(x) - k_{x_i}(x_i)} < \varepsilon$ for all $x_i \in U$. Since $x_i$ is eventually in $U$ and $\lim_{i \to \infty} k_{x_i}(x) = k(x,x)$ then $\abs{k(x,x) - k_{x_i}(x_i)} < \varepsilon$ for all $x_i \in U$. Therefore the kernel is jointly weak-$*$ continuous and so $\lim_{i \to \infty} \norm{k_x - k_{x_i}}^2 = \lim_{i \to \infty} k(x,x)+k(x_i,x_i)-k(x,x_i)-k(x_i,x) = 0$.
\end{proof}

\subsection*{Acknowledgments} DG acknowledges support from Basic Research Office of the U.S.\ Department of Defense under Vannevar Bush Faculty Fellowship grant N00014-21-1-2946. MM was supported as a postdoctoral fellow under the same grant. The authors thank Hans Feichtinger for feedback on an earlier draft of this paper.

\pagebreak

\bibliographystyle{alpha}
\bibliography{bibliography}

\begin{thebibliography}{EGNO15}

\bibitem[Aro50]{Aronszajn50}
N.~Aronszajn.
\newblock Theory of reproducing kernels.
\newblock {\em Trans. Amer. Math. Soc.}, 68(3):337--404, 1950.

\bibitem[BPTV19]{BrunoEtAl19}
T.~Bruno, M.~Peloso, A.~Tabacco, and M.~Vallarino.
\newblock Sobolev spaces on {L}ie groups: {E}mbedding theorems and algebra
  properties.
\newblock {\em J. Funct. Anal.}, 276(10):3014--3050, 2019.

\bibitem[Bra75]{Brandenburg75}
L.~H. Brandenburg.
\newblock On identifying the maximal ideals in {B}anach algebras.
\newblock {\em J. Math. Anal. App.}, 50(3):489--510, 1975.

\bibitem[Cho78]{Cho78}
Wojciech Chojnacki.
\newblock On {B}anach algebras which are {H}ilbert spaces.
\newblock {\em Ann. Soc. Math. Polon. Ser. I Comment. Math. Prace Mat.},
  20(2):279--281, 1978.

\bibitem[CNW73]{CNW73}
J.~Chover, P.~Ney, and S.~Wainger.
\newblock Functions of probability measures.
\newblock {\em J. Anal. Math.}, 26:255--302, 1973.

\bibitem[Del97]{Delvos97}
F.~J. Delvos.
\newblock Interpolation in harmonic {H}ilbert spaces.
\newblock {\em Math. Model. Numer. Anal.}, 31(4):435--458, 1997.

\bibitem[DG23]{DG23}
S.~Das and D.~Giannakis.
\newblock On harmonic {H}ilbert spaces on compact abelian groups.
\newblock {\em J. Fourier Anal. Appl.}, 29(1):12, 2023.

\bibitem[DGM23]{DGM23}
S.~Das, D.~Giannakis, and M.~Montgomery.
\newblock Correction to: On harmonic {H}ilbert spaces on compact abelian
  groups.
\newblock {\em J. Fourier Anal. Appl.}, 29(6):67, 2023.

\bibitem[Dom56]{Domar56}
Y.~Domar.
\newblock Harmonic analysis based on certain commutative {B}anach algebras.
\newblock {\em Acta Math.}, 96:1--66, 1956.

\bibitem[EGNO15]{EGNO15}
P.~Etingof, S.~Gelaki, D.~Nikshych, and V.~Ostrik.
\newblock {\em Tensor Categories}.
\newblock American Mathematical Society, 2015.

\bibitem[Fei79]{F79}
H.~G. Feichtinger.
\newblock Gewichtsfunktionen auf lokalkompakten {G}ruppen.
\newblock {\em {\"O}sterreich. Akad. Wiss. Math.-Natur. Kl. Sitzungsber. II},
  188(8--10):451--471, 1979.

\bibitem[FGL06]{FGL06}
G.~Fendler, K.~Gr{\"o}chenig, and M.~Leinert.
\newblock Symmetry of weighted {$L^1$}-algebras and the {GRS}-condition.
\newblock {\em Bull. London Math. Soc.}, 38(4):529--704, 2006.

\bibitem[FPW07]{FeichtingerEtAl07}
H.~G. Feichtinger, S.~S. Pandey, and T.~Werther.
\newblock Minimal norm interpolation in harmonic {H}ilbert spaces and {W}iener
  amalgam spaces on locally compact abelian groups.
\newblock {\em J. Math. Kyoto Univ.}, 47(1):65--78, 2007.

\bibitem[Gr{\"o}07]{Grochenig07}
K.~Gr{\"o}chenig.
\newblock Weight functions in time-frequency analysis.
\newblock In L.~Rodino et~al., editors, {\em Pseudodifferential Operators:
  Partial Differential Equations and Time-Frequency Analysis}, volume~52 of
  {\em Fields Inst. Commun.}, pages 343--366. American Mathematical Society,
  Providence, 2007.

\bibitem[HR79]{HR79}
E.~Hewitt and K.~Ross.
\newblock {\em Abstract Harmonic Analysis}, volume~1.
\newblock Springer Verlag, Berlin, 1979.

\bibitem[Ing34]{Ing34}
A.~Ingham.
\newblock A note on {F}ourier transforms.
\newblock {\em J. London Math. Soc.}, 9(33):29--32, 1934.

\bibitem[Kan09]{Kan09}
E.~Kaniuth.
\newblock {\em A Course in Commutative {B}anach Algebras}, volume 246 of {\em
  Graduate Texts in Mathematics}.
\newblock Springer Science+Media, 2009.

\bibitem[Kuz06]{Kuznetsova06}
Y.~Kuznetsova.
\newblock Weighted {$L^p$}-algebras on groups.
\newblock {\em Funct. Anal. Appl.}, 40(3):234--236, 2006.

\bibitem[Nik70]{Nik70}
N.~Nikolskii.
\newblock Spectral synthesis for the shift operator, and zeros in certain
  classes of analytic functions that are smooth up to the boundary.
\newblock {\em Dokl. Akad. Nauk SSSR}, 190(4):780–783, 1970.

\bibitem[PR16]{PR16}
V.~Paulsen and M.~Raghupathi.
\newblock {\em An Introduction to the Theory of Reproducing Kernel Hilbert
  Spaces}.
\newblock Cambridge University Press, Cambridge, 2016.

\bibitem[Tch84]{Tchamitchian84}
P.~Tchamitchian.
\newblock G{\`e}neralization des alg{`e}bres de {B}eurling.
\newblock {\em Ann. Inst. Fourier}, 34(4):151--168, 1984.

\bibitem[Tch87]{Tchamitchian87}
P.~Tchamitchian.
\newblock {\'Etude} dans un cadre hilbertien des alg{\`e}bres de {B}eurling
  munies d'un poids radial {\`a} croissance rapide.
\newblock {\em Ark. Mat.}, 25(1--2):295--312, 1987.

\end{thebibliography}

\end{document}